\documentclass{article}

\usepackage{amsmath,amsthm,amssymb}
\usepackage{graphicx}
\usepackage{epstopdf}

\DeclareMathOperator{\diag}{diag}
\DeclareMathOperator{\arccot}{arccot}
\DeclareMathOperator{\Res}{Res}
\DeclareMathOperator{\sgn}{sgn}
\DeclareMathOperator{\CH}{CH}
\DeclareMathOperator{\SH}{SH}
\DeclareMathOperator{\CC}{CC}
\DeclareMathOperator{\SC}{SC}

\newtheorem{lemma}{Lemma}[section]
\newtheorem{proposition}{Proposition}[section]
\newtheorem{theorem}{Theorem}[section]

\newcommand{\M}{\mathbf M}

\theoremstyle{remark}
\newtheorem*{remark}{Remark}

\begin{document}

\title{\LARGE Numerical Approximation of the Fractional Laplacian on $\mathbb R$ Using Orthogonal Families}

\author{\normalsize Jorge Cayama$^1$ \and \normalsize Carlota M. Cuesta$^1$ \and \normalsize Francisco de la Hoz$^2$}

\date{
	\footnotesize
	$^1$Department of Mathematics, Faculty of Science and Technology, University of the Basque Country UPV/EHU, Barrio Sarriena S/N, 48940 Leioa, Spain \\[1em]
	$^2$Department of Applied Mathematics and Statistics and Operations Research, Faculty of Science and Technology, University of the Basque Country UPV/EHU, Barrio~Sarriena~S/N, 48940 Leioa, Spain
}

\maketitle

\begin{abstract}

In this paper, using well-known complex variable techniques, we compute explicitly, in terms of the ${}_2F_1$ Gaussian hypergeometric function, the one-dimensional fractional Laplacian of the Higgins functions, the Christov functions, and their sine-like and cosine-like versions. After discussing the numerical difficulties in the implementation of the proposed formulas, we develop a method using variable precision arithmetic that gives accurate results. 

\end{abstract}

\medskip

\noindent\textit{Keywords:}

\noindent Fractional Laplacian, Pseudospectral methods, rational Chebyshev functions, variable precision arithmetic
	
\section{Introduction}

In this paper, we obtain explicit expressions for the one dimensional fractional Laplacian of the Higgins and Christov functions. We also explain how to implement these results efficiently in \textsc{Matlab} \cite{matlab}, and give numerical examples as an application. More precisely, we work with the following definition \cite{kwasnicki} of the fractional Laplacian:
\begin{equation}
\label{e:fraclap0}
(-\Delta)^{\alpha/2}u(x) = c_\alpha\int_{-\infty}^\infty\frac{u(x)-u(x+y)}{|y|^{1+\alpha}}dy, \text{ with }c_\alpha = \alpha\frac{2^{\alpha-1}\Gamma(1/2+\alpha/2)}{\sqrt{\pi}\Gamma(1-\alpha/2)}, \ \alpha\in(0,2).
\end{equation}

\noindent We recall that the Fourier transform of $u$ is $\hat u(\xi) = \int_{-\infty}^{\infty}u(x)e^{-ix\xi}dx$, and that the inverse Fourier transform is $u(x) = \frac{1}{2\pi}\int_{-\infty}^{\infty}\hat u(\xi)e^{ix\xi}d\xi$; hence, another definition of the fractional Laplacian consistent with \eqref{e:fraclap0} is given by associating the operator with the Fourier symbol:
\begin{equation}\label{fourier:sym}
[(-\Delta)^{\alpha/2}u]^\wedge(\xi) = |\xi|^\alpha \hat u(\xi).
\end{equation}

\noindent We note that, when $\alpha=0$, according to \eqref{fourier:sym}, $(-\Delta)^0u(x) = u(x)$. We also observe, however, that, when the Fourier transform of $u(x)$ does not exist in the classical sense, the limit $\alpha\to0^+$ is singular. For instance, take $u(x) = 1$, whose Fourier transform is the Dirac delta distribution, i.e., $\hat u(\xi) = 2\pi\delta(\xi)$; then
\begin{equation}
\label{e:discontinuousa0}
(-\Delta)^\alpha 1 = 
\begin{cases}
1, & \alpha = 0,
	\cr
0, & \alpha > 0.
\end{cases}
\end{equation}

\noindent On the other hand, when $\alpha=2$, $(-\Delta)^{2/2} u(x) = -u_{xx}(x)$; and, when $\alpha = 1$, $\widehat{(-\Delta)^{1/2}u}(\xi) = |\xi|\hat u(\xi) = -i\sgn(\xi)(i\xi)\hat u(\xi) = \widehat{\mathcal H(u_x)}(\xi)$, where $\mathcal H$ denotes the Hilbert transform \cite{hilbert}:
\begin{equation}\label{H:trans}
\mathcal H (u)(x) = \frac1\pi\int_{-\infty}^{\infty}\frac{u(y)}{x-y}dy,
\end{equation}

\noindent which has the Fourier symbol $\widehat{\mathcal H(u)}(\xi) = -i\sgn(\xi)\hat u(\xi)$. Thus, formally, $(-\Delta)^{1/2}u = \mathcal H(u_x)$.

The fractional Laplacian appears in a number of applications (see, for instance, \cite[Table 1]{lischke} and the references therein). In what regards its numerical computation, most references in the literature involve truncation of the domain (see for instance the works \cite{IlicLiuTurnerAnh2005}, \cite{YangLiuTurner2010}, and \cite{Bueno-OrovioKayBurrage2012}, where, although a
truncation of the domain is not explicitly given, the method relays on the spectral definition of the fractional Laplacian on a bounded domain). However, in a recent paper \cite{cayamacuestadelahoz2019}, we proposed a pseudospectral method to compute the fractional Laplacian of a bounded function $u(x)$ on $\mathbb R$ without truncation. The idea was to map $\mathbb R$ into the finite interval $[0,\pi]$ by means of the change of variable $x = L\cot(s)$, and then obtain the trigonometric Fourier series expansion of $u(x(s))$. Therefore, the central point of \cite{cayamacuestadelahoz2019} was the efficient and accurate numerical computation of the fractional Laplacian of $e^{ins}$, for $n\in\mathbb Z$; observe that the resulting expressions for $(-\Delta)^{1/2}e^{ins}$ are quite different, depending on whether $n$ is even or odd. At this point, a crucial observation is that the equivalent on $\mathbb R$ of the functions $e^{i2ns}$ (i.e., the even case) are precisely the complex Higgins functions defined in \cite{boyd1990} (see also \cite{higgins,narayan}):
\begin{equation}
\label{e:ln}
\lambda_n(x) = \left(\frac{ix - 1}{ix + 1}\right)^n, \quad n\in\mathbb Z,
\end{equation}

\noindent because $\lambda_n(\cot(s)) = e^{i2ns}$. Indeed, in this paper, thanks to the structure of $\lambda_n(x)$, and after rewriting \eqref{e:fraclap0} as
\begin{equation}
\label{e:fraclapl}
(-\Delta)^{\alpha/2}u(x) = \frac{c_{\alpha}}{\alpha}\int_{0}^\infty \frac{u_{x}(x-y)-u_{x}(x+y)}{y^{\alpha}}dy,
\end{equation}

\noindent we obtain an explicit expression of their fractional Laplacian by means of contour integration:
\begin{equation}
\label{e:fraclapln}
(-\Delta)^{\alpha/2}\lambda_n(x) =
\begin{cases}
0, & n = 0,
\\
-\displaystyle{\frac{2|n|\Gamma(1+\alpha)}{(i\sgn(n)x+1)^{1+\alpha}}{}_2F_1\left(1-|n|,1+\alpha; 2;\frac{2}{i\sgn(n)x+1}\right)}, & n\in\mathbb Z\backslash \{0\},
\end{cases}
\end{equation}

\noindent where ${}_2F_1$ is the Gaussian hypergeometric function (see for instance \cite[Ch. 15]{abramowitz}). From this result, we also derive several other related ones outlined below.

The structure of this paper is as follows. In Section \ref{s:fraclapln}, we prove \eqref{e:fraclapln}, which constitutes the main result. We observe that $\{\lambda_n(x)\}$ is a complete orthogonal system in $\mathcal L^2(\mathbb R)$ with weight $w(x) = 1/(\pi(1+x^2))$, because $\{e^{i2ns}\}$ is a complete orthonormal system in $\mathcal L^2([0,\pi])$, normalized by $w= 1/\pi$. Therefore, the related family of functions
\begin{equation}
\label{e:mun}
\mu_n(x) = \frac{(ix - 1)^n}{(ix + 1)^{n+1}}, \quad n\in\mathbb Z,
\end{equation}

\noindent known as the complex Christov functions \cite{boyd1990,christov,narayan,wiener}, form a complete orthogonal system in $\mathcal L^2(\mathbb R)$ (normalized by the factor $1/\pi$). 

Throughout this paper, we use the definition and notation from \cite{boyd1990} for the following families of functions:
\begin{itemize}
\item Cosine-like Higgins functions:
\begin{equation}
\label{e:CH}
\CH_{2n}(x) = \frac{\lambda_n(x) + \lambda_{-n}(x)}2, \quad n = 0, 1, 2, \ldots
\end{equation}

\item Sine-like Higgins functions:
\begin{equation}
\label{e:SH}
\SH_{2n+1}(x) = \frac{\lambda_{n+1}(x) - \lambda_{-n-1}(x)}{2i}, \quad n = 0, 1, 2, \ldots
\end{equation}

\item Cosine-like Christov functions:
\begin{equation}
\label{e:CC}
\CC_{2n}(x) = \frac{\mu_n(x) - \mu_{-n-1}(x)}2, \quad n = 0, 1, 2, \ldots
\end{equation}

\item Sine-like Christov functions:
\begin{equation}
\label{e:SC}
\SC_{2n+1}(x) = -\frac{\mu_n(x) + \mu_{-n-1}(x)}{2i}, \quad n = 0, 1, 2, \ldots
\end{equation}
\end{itemize}

\noindent In Section \ref{s:fraclapmore}, starting from \eqref{e:fraclapln}, we calculate the fractional Laplacian of \eqref{e:mun}-\eqref{e:SC}.

Even if the fractional Laplacian of all the families considered here can be computed accurately with the technique explained in \cite{cayamacuestadelahoz2019}, expressions like \eqref{e:fraclapln} have the advantage of being very compact and, hence, it is effortless to use them in numerical applications, provided that fast accurate implementations of the Gaussian hypergeometric function ${}_2F_1$ are available. Therefore, in Section \ref{s:numerical}, using \textsc{Matlab}, we test their adequacy from a numerical point of view, comparing the numerical results with those in \cite{cayamacuestadelahoz2019}. On the one hand, for moderately large values of $n$, the use of variable precision arithmetic seems unavoidable; on the other hand, our implementation of ${}_2F_1$ largely outperforms that of \textsc{Matlab}. Finally, even if the method in \cite{cayamacuestadelahoz2019} is faster, the method in this paper is much easier to implement and still competitive for not too large values of $n$.

\section{Fractional Laplacian of the complex Higgins functions}

\label{s:fraclapln}

Before we proceed, let us recall some well-known definitions. Given $z\in\mathbb C$, the generalized binomial coefficient is defined by
\begin{equation*}
\binom{z}{n} = 
\begin{cases}
\dfrac{z(z-1)\ldots(z-n+1)}{n!}, & n\in\mathbb N,
\\
1, &  n = 0.
\end{cases}
\end{equation*}

\noindent where $\mathbb N = \{1, 2, 3, \ldots\}$. Therefore, if $z$ is a nonnegative integer, and $n > z$, then $\binom{z}{n} = 0$. Furthermore, it is immediate to check that, for all $z$ and $n$,
$$
\binom{z}{n} = (-1)^n\binom{n-1-z}{n}.
$$

\noindent We will also need the Pochhammer symbol, which represents the rising factorial, and is defined by
$$
(z)_n =
\begin{cases}
z(z+1)\ldots(z+n-1), & n\in\mathbb N,
\\
1, & n = 0.
\end{cases}
$$

\noindent Observe that, when $z$ is not zero or a negative integer, an equivalent definition is
$$
(z)_n = \frac{\Gamma(z+n)}{\Gamma(z)};
$$

\noindent in particular, when $z\in\mathbb N$,
$$
(z)_n = \frac{(z+n-1)!}{(z-1)!}.
$$

\noindent Remark that, if $z$ is a negative integer or zero, and $n > |z|$, then $(z)_n = 0$. Moreover, the following identities will be useful, too:
\begin{align*}
(-z)_n & = (-1)^n(z-n+1)_n,
\cr
\binom{z}{n} & = \frac{(z-n+1)_n}{n!} = \frac{(-1)^n(-z)_n}{n!}.
\end{align*}

\noindent The Pochhammer symbol also appears in the definition of the Gaussian hypergeometric function ${}_2F_1$ (see for instance \cite[Ch. 15]{abramowitz}). Let $a, b, c, z\in\mathbb C$; then, ${}_2F_1$  is defined by 
\begin{equation}
\label{e:2F1}
{}_2F_1(a, b; c; z) = \sum_{k=0}^{\infty}\frac{(a)_k(b)_k}{(c)_k}\frac{z^k}{k!}.
\end{equation}

\noindent In general, the infinite series converges for $|z|< 1$. However, in our case, we take $a$ to be a negative integer, so the sum is finite, because of the properties of the Pochhammer symbol. More precisely, in this paper, we are interested in the following two particular cases:
\begin{align}
    {}_2F_1(-m, 1+\alpha; 1; z)  =
    \sum_{k=0}^m\binom{m}{k}\binom{-1-\alpha}{k}z^k,
\label{hyper:1}\\
  {}_2F_1(-m, 1+\alpha; 2; z)  =
  -\frac{1}{\alpha}\sum_{k=0}^m\binom{m}{k}\binom{-\alpha}{k+1}z^k = \frac{1}{m+1}\sum_{k=0}^m\binom{m+1}{k+1}\binom{-1-\alpha}{k}z^k,
\label{hyper:2}
\end{align}

\noindent for $m\in\mathbb N$. Observe that the identities also hold after replacing $k$ by $m-k$ in the sums, a fact that we will also use below. On the other hand, if $m = 0$, ${}_2F_1(0, 1+\alpha; 1; z) = {}_2F_1(0, 1+\alpha; 2; z) = 1$.

Bearing in mind the previous arguments, let us prove \eqref{e:fraclapln}, which is the main result of this paper. Note that we work with the binomial coefficient rather than with the Pochhammer symbol, because we think that the former is more intuitive.

\begin{theorem}

\label{t:theo1}
	
Let $\lambda_n(x)$ be defined as in \eqref{e:ln}. Then, $(-\Delta)^{\alpha/2}\lambda_n(x)$ is given by \eqref{e:fraclapln}.

\end{theorem}

\begin{proof}

The case with $n = 0$ is trivial, because $\lambda_0(x) = 1$. Assume $n\in\mathbb N$. The derivative of \eqref{e:ln} is
$$
\lambda_n'(x) = -\frac{2ni}{(x - i)^2}\left(\frac{x + i}{x - i}\right)^{n-1}.
$$

\noindent Introducing this expression in \eqref{e:fraclapl}, we get
\begin{equation}
\label{e:fraclapln1}
(-\Delta)^{\alpha/2}\lambda_n(x) = 2ni\frac{2^{\alpha-1}\Gamma(1/2+\alpha/2)}{\sqrt{\pi}\Gamma(1-\alpha/2)}\int_{0}^\infty g_n(y;x)dy,
\end{equation}

\noindent with integrand
$$
g_{n}(z;x) = \frac{(z+x+i)^{n-1}(z-x+i)^{n+1} - (z-x-i)^{n-1}(z+x-i)^{n+1}}{z^{\alpha}(z+x-i)^{n+1}(z-x+i)^{n+1}},
$$

\noindent where, in this notation, we regard $x$ as a parameter, rather than as an independent variable.

We next compute \eqref{e:fraclapln1} for every $x$, by integrating $g_{n}(z;x)$ along certain integration contours $C$ in $\mathbb{C}$, and using Cauchy's integral theorem. Since $z^{\alpha} = e^{\alpha\ln(z)} = e^{\alpha(\ln(z) + i\arg(z))}$, $z^{\alpha}$ has a branch cut. In what follows, we consider the principal branch of the logarithm, which corresponds to $-\pi < \arg(z) \le \pi$; in particular, $(-1)^\alpha = e^{i\pi\alpha}$, unless the branch cut is crossed. The branch choice determines also how we choose the contours. In Figure \ref{f:contour}, we have depicted one such contour $C$ for $x>0$, which consists of four parts, avoids the branch cut, but encloses the poles $z = i - x$ and $z = x - i$, for every $x$. Then, by the residue theorem, the integral along it is equal to the sum of the residues. The pieces of the contour that run parallel to the branch cut will give the approximation of the integral from $0$ to $\infty$; the other pieces will give integrals that tend to zero, when $C$ tends to one contour that encloses $\mathbb{C}$, except for the brach cut. More precisely, for every $x\in\mathbb{R}$, we take $R>0$, such that $R\gg (1+|x|^2)^{1/2}$, and $r>0$, such that $r\ll (1+|x|^2)^{1/2}$, for instance, $r=1/R$. We also take $\delta>0$, such that, with $\theta_1\in(-\pi,-\pi/2)$ and $\theta_2\in(\pi/2,\pi)$ fixed, $\delta=r\sin\theta_2$. Then, we define
$$
C = C_1 \cup C_R\cup C_2\cup C_r,
$$

\noindent with
\begin{align*}
  C_1 & = \{ -y-i\delta : y\in (-r\cos\theta_1,-R\cos\theta_1)\},\\
  C_R & = \{Re^{\theta i}: \ \theta \in (\theta_1,\theta_2)\},\\
  C_2 & = \{ y+i\delta : y\in (R\cos\theta_2,r\cos\theta_2)\},\\
  C_r & = \{ re^{-\theta i}: \ \theta \in (-\theta_2,-\theta_1)\}.
\end{align*}

\noindent Here, we have omitted the dependecy on $x$ for simplicity of notation. Now, by Cauchy's residue theorem, we have:
\begin{align}
\label{e:intgg}
\int_C g_{n}(z;x)dz & = \int_{C_1}g_{n}(z;x)dz + \int_{C_r}g_{n}(z;x)dz + \int_{C_2}g_{n}(z;x)dz + \int_{C_R}g_{n}(z;x)dz
	\cr
& = 2\pi i [\Res(g_{n}(z;x),i-x) + \Res(g_{n}(z;x),x-i)].
\end{align}

\begin{figure}[!htbp]
	\centering
	\includegraphics[width=0.5\textwidth, clip=true]{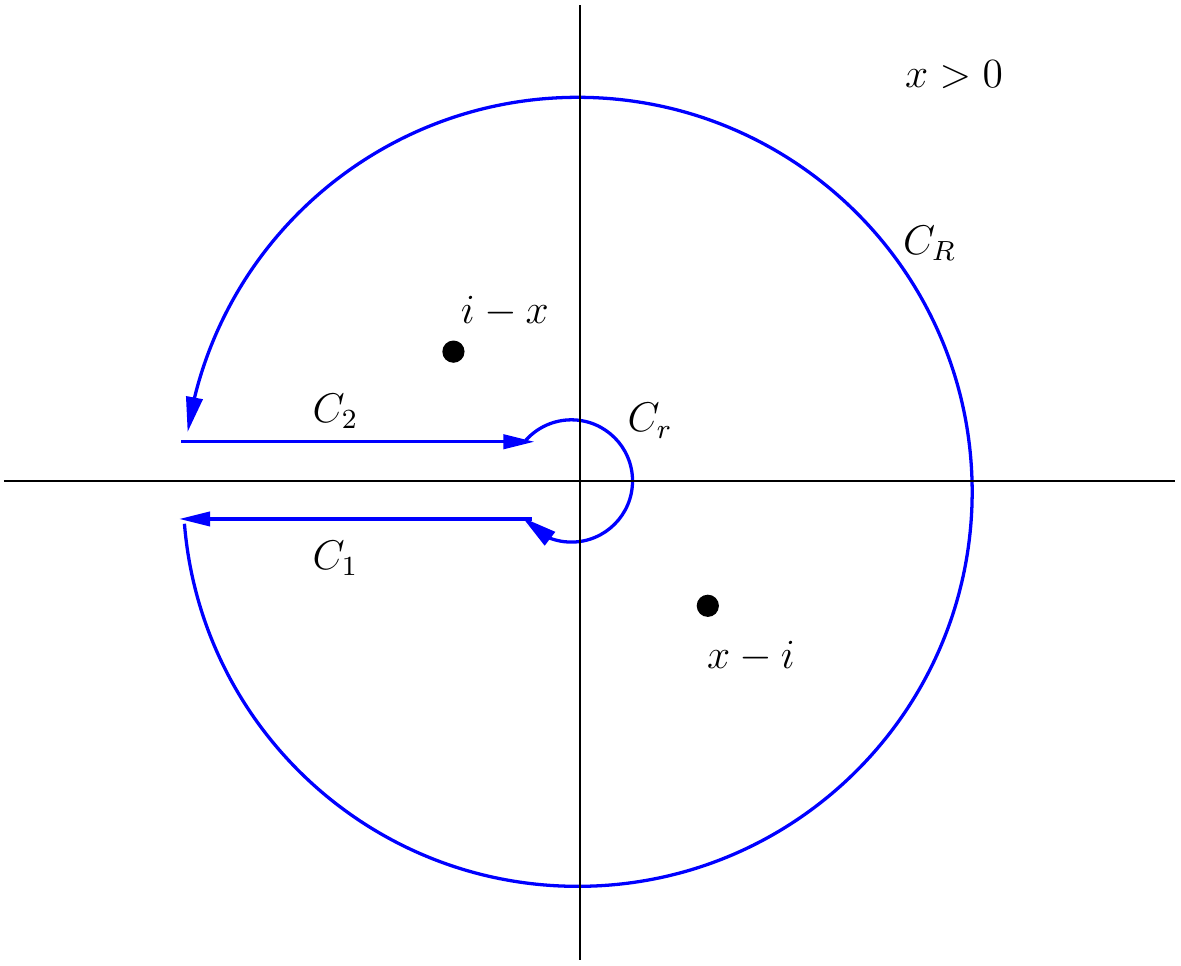}
	\caption{An example of integration contour for $x>0$.}
	\label{f:contour}	
\end{figure}

\noindent In order to compute the residues of $g_{n}(z;x)$ at $z = i - x$ and $z = x - i$, we use the general Leibniz rule:
$$
\frac{d^n}{dz^n}(p(z)q(z)) = \sum_{k=0}^{n}\binom{n}{k}p^{(k)}(z)q^{(n-k)}(z).
$$

\noindent In this way, we obtain
\begin{align}
\label{e:resi-x}
\Res(g_{n}(z;x),i-x) & = \left.\frac{1}{n!}\frac{d^n}{dz^n}\left(\frac{(z+x+i)^{n-1}(z-x+i)^{n+1} - (z-x-i)^{n-1}(z+x-i)^{n+1}}{z^{\alpha}(z-x+i)^{n+1}}\right)\right|_{z = i-x}
\cr
& = \left.\frac{1}{n!}\frac{d^n}{dz^n}z^{-\alpha}(z+x+i)^{n-1}\right|_{z = i-x}
\cr
& = \left.\frac{1}{n!}\sum_{k = 1}^n\binom{n}{k}\left[k!\binom{-\alpha}{k}z^{-\alpha-k}\right]\left[\frac{(n-1)!}{(n-1-(n-k))!}(z+x+i)^{n-1-(n-k)}\right]\right|_{z = i-x}
\cr
& = \sum_{k = 1}^n\binom{n-1}{k-1}\binom{-\alpha}{k}(i-x)^{-\alpha-k}(2i)^{k-1}.
\end{align}

\noindent Likewise, we have
\begin{equation}
\label{e:resx-i}
\Res(g_{n}(z;x),x-i) = -\sum_{k = 1}^n\binom{n-1}{k-1}\binom{-\alpha}{k}(x-i)^{-\alpha-k}(-2i)^{k-1}.
\end{equation}

\noindent In order to compute \eqref{e:intgg}, we observe that the first and third integrals tend to zero, as $R$ tends to infinity (this can be easily seen by changing to polar coordinates and recalling that $r=1/R$). In what regards the first and third integrals, we have that $C_1$ tends to $(-\infty,0)$, parameterized by $-y$, with $y\in(0,\infty)$, and $C_2$ tends to $(-\infty,0)$, parameterized by $y$, with $y\in(-\infty,0)$. We observe that $g_n(-y;x)=(-1)^{1-\alpha} g_n(y;x)$, where the argument of $-1$ is determined by the curve $C_i$ before taking the limit $R\to \infty$. Thus, it is $-\pi$ on $C_1$, and $\pi$ on $C_2$. Then,
\begin{align*}
  \lim_{R\to \infty}\int_{C_1}g_{n}(z;x)dz=-\int_{0}^\infty g_{n}(-y;x)dy = e^{i\pi\alpha}\int_{0}^\infty g_n(y;x) dy,
\end{align*}
and
\begin{align*}
  \lim_{R\to \infty}\int_{C_2}g_{n}(z;x)dz & = \int_{-\infty}^0 g_{n}(y;x)dy=
  \int_0^\infty g_{n}(-y;x)dy=-e^{-i\pi\alpha}\int_{0}^\infty g_n(y;x) dy.
\end{align*}

\noindent Therefore,
$$
\int_C g_{n}(z;x)dz = (e^{i\pi \alpha} - e^{-i\pi \alpha})\int_0^\infty g_{n}(y;x)dy = 2\pi i[\Res(g_{n}(z;x),i-x) + \Res(g_{n}(z;x),x-i)].
$$

\noindent Finally, after some manipulations, we have
\begin{align*}
\int_0^\infty g_{n}(y;x)dy & = \frac{2\pi i}{2i\sin(\pi\alpha)}[\Res(g_{n}(z;x),i-x) + \Res(g_{n}(z;x),x-i)]
\\
& = \frac{\pi}{\sin(\pi\alpha)}\sum_{k = 1}^n\binom{n-1}{k-1}\binom{-\alpha}{k}[(i-x)^{-\alpha-k}(2i)^{k-1} - (x-i)^{-\alpha-k}(-2i)^{k-1}]
\\
& = -\,\frac{\pi((-1)^{-\alpha} + 1)}{\sin(\pi\alpha)}\sum_{k = 1}^n\binom{n-1}{k-1}\binom{-\alpha}{k}(x-i)^{-\alpha-k}(-2i)^{k-1}
\\
& = -\,\frac{\pi(e^{-i\pi\alpha} + 1)}{\sin(\pi\alpha)(x-i)^{1+\alpha}}\sum_{k = 0}^{n-1}\binom{n-1}{k}\binom{-\alpha}{k+1}\left(\frac{-2i}{x-i}\right)^k
\\
& = \frac{\pi \alpha e^{-i\pi\alpha/2} (i)^{1+\alpha}}{\sin(\pi\alpha/2)(ix+1)^{1+\alpha}}{}_2F_1\left(1-n, 1+\alpha, 2, \frac{2}{ix + 1}\right),
\end{align*}

\noindent where we have used (\ref{hyper:2}). Substituting this last expression into \eqref{e:fraclapln1}, and using Euler's reflection formula,
$$
\Gamma(w) \Gamma(1 - w) = \frac{\pi}{\sin(\pi w)},
$$
\noindent at $w = \alpha/2$, we obtain
\begin{equation}
\label{e:lnpos}
(-\Delta)^{\alpha/2}\lambda_n(x) = 2ni\frac{2^{\alpha-1}\Gamma(1/2+\alpha/2)}{\sqrt{\pi}\Gamma(1-\alpha/2)}\frac{i\alpha\Gamma(\alpha/2)\Gamma(1-\alpha/2)}{(ix + 1)^{1+\alpha}}{}_2F_1\left(1-n, 1+\alpha, 2, \frac{2}{ix+1}\right).
\end{equation}

\noindent Then, \eqref{e:fraclapln} for $n\in\mathbb{N}$ follows from \eqref{e:lnpos}, using Legendre's duplication formula $\Gamma(w)\Gamma(w + 1/2) = 2^{1-2w}\sqrt\pi\Gamma(2w)$, at $w = \alpha/2$. In order to finish the proof, we notice that the case where $n$ is a negative integer follows from the symmetry
\begin{equation}
\label{e:lnpm}
\lambda_{n}(x) = \overline{\lambda_{-n}(x)}, \quad n\in\mathbb N.
\end{equation}

\end{proof}

\section{Fractional Laplacian of other families of functions}

\label{s:fraclapmore}

In this section, we compute the fractional Laplacian of the families of functions defined by \eqref{e:mun}-\eqref{e:SC}. Let us obtain first the Fractional Laplacian of the complex Christov Functions:
\begin{proposition}
Let $\mu_n(x)$ be defined as in \eqref{e:mun}. Then,
\begin{equation}
  \label{e:fraclapmun}
(-\Delta)^{\alpha/2}\mu_n(x) =
\begin{cases}
\dfrac{\Gamma(1+\alpha)}{(ix+1)^{1+\alpha}}{}_2F_1\left(-n,1+\alpha; 1;\dfrac{2}{ix+1}\right), & n = 0, 1, 2, \ldots,
	\\[1em]
-\dfrac{\Gamma(1+\alpha)}{(-ix+1)^{1+\alpha}}{}_2F_1\left(1+n,1+\alpha; 1;\dfrac{2}{-ix+1}\right), & n = -1, -2, -3, \ldots
\end{cases}
\end{equation}

\noindent Moreover, the expression for $n=0$ reduces to 
$$
(-\Delta)^{\alpha/2}\mu_{0}(x) = \frac{\Gamma(1+\alpha)}{(ix+1)^{1+\alpha}}.
$$
\end{proposition}

\begin{proof}

First, we note that $\mu_n(x)$ can be expressed in terms of $\lambda_n(x)$ as follows:
$$
\mu_n(x) = \lambda_n(x)\frac{1}{(ix+1)} =\frac{1}{2}\lambda_n(x)\frac{(ix+1)-(ix-1)}{(ix+1)} = \frac{\lambda_n(x) - \lambda_{n+1}(x)}2;
$$

\noindent therefore,
\begin{equation}
\label{e:mulambda}
(-\Delta)^{\alpha/2}\mu_{n}(x) = \frac{(-\Delta)^{\alpha/2}\lambda_{n}(x) - (-\Delta)^{\alpha/2}\lambda_{n+1}(x)}2,
\end{equation}

\noindent so we just have to use \eqref{e:fraclapln} and simplify the resulting expression. First we assume that $n\in\mathbb N$. Substituting \eqref{e:fraclapln} into this last equation, and using that $\binom{n+1}{k+1} = \binom{n}{k+1} + \binom{n}{k}$, we get
\begin{align*}
(-\Delta)^{\alpha/2}\mu_n(x) & = \frac{\Gamma(1+\alpha)}{(ix+1)^{1+\alpha}}\left[\sum_{k = 0}^n\binom{n+1}{k+1}\binom{-1-\alpha}{k}\left(\frac{2}{ix+1}\right)^k - \sum_{k = 0}^{n-1}\binom{n}{k+1}\binom{-1-\alpha}{k}\left(\frac{2}{ix+1}\right)^k\right]
\cr
& = \frac{\Gamma(1+\alpha)}{(ix+1)^{1+\alpha}}\sum_{k = 0}^n\binom{n}{k}\binom{-1-\alpha}{k}\left(\frac{2}{ix+1}\right)^k,
\end{align*}

\noindent which is \eqref{e:fraclapmun}, for $n\in\mathbb N$. On the other hand, when $n = 0$, again from \eqref{e:fraclapln},
$$
(-\Delta)^{\alpha/2}\mu_{0}(x) = -\frac{(-\Delta)^{\alpha/2}\lambda_{1}(x)}2 = \frac{\Gamma(1+\alpha)}{(ix+1)^{1+\alpha}}{}_2F_1\left(0,1+\alpha; 2;\frac{2}{ix+1}\right) = \frac{\Gamma(1+\alpha)}{(ix+1)^{1+\alpha}},
$$

\noindent and \eqref{e:fraclapmun} also holds. Finally, when $n$ is a negative integer, we observe that
\begin{equation}
\label{e:munpm}
\mu_{n}(x) = -\overline{\mu_{-1-n}(x)},
\end{equation}

\noindent so, in that case, we use $(-\Delta)^{\alpha/2}\mu_{n}(x) = -\overline{(-\Delta)^{\alpha/2}\mu_{-1-n}(x)}$, which concludes the proof.

\end{proof}

Before we compute the fractional Laplacian of the cosine-like and sine-like Higgins functions \eqref{e:CH} and \eqref{e:SH}, we first express \eqref{e:fraclapln} as a polynomial on $x$ times a negative power of $(i\sgn(n)x+1)$.
\begin{lemma}
	
\label{p:lemmaln}
	
Let $\lambda_n(x)$ be defined as in \eqref{e:ln}, and $n\in\mathbb Z\backslash\{0\}$. Then, \eqref{e:fraclapln} can be expanded as
\begin{equation}
\label{e:fraclaplnexpan}
(-\Delta)^{\alpha/2}\lambda_n(x) = -\frac{2|n|\Gamma(1+\alpha)}{(i\sgn(n)x+1)^{|n|+\alpha}}\sum_{l=0}^{|n|-1}\binom{|n|-1}{l}(i\sgn(n)x)^{|n|-1-l}{}_2F_1(-l,1+\alpha; 2; 2).
\end{equation}

\end{lemma}

\begin{proof}

Assume first that $n\in\mathbb N$. We express ${}_2F_1$ in \eqref{e:lnpos} as a sum; then, replacing the index $k$ by $n - 1 - k$ in the sum, and expanding $(ix+1)^l$ by Newton's binomial formula, we get
\begin{align}
\label{e:fraclaplnx0}
(-\Delta)^{\alpha/2}\lambda_n(x) & = -\frac{2\Gamma(1+\alpha)}{(ix+1)^{1+\alpha}}\sum_{k=0}^{n-1}\binom{n}{k+1}\binom{-1-\alpha}{k}\left(\frac{2}{ix+1}\right)^k
	\cr
& = -\frac{2^n\Gamma(1+\alpha)}{(ix+1)^{n+\alpha}}\sum_{k=0}^{n-1}\sum_{l=0}^k\frac{1}{2^k}\binom{n}{k}\binom{k}{l}\binom{-1-\alpha}{n-1-k}(ix)^l.
\end{align}

\noindent We now interchange the order of the sums and replace the indices $l$ by $n-1-l$ and $k$ by $n-1-k$; then, after some rewriting, we get,
\begin{align}
\label{e:fraclaplnx1}
(-\Delta)^{\alpha/2}\lambda_n(x) & = -\frac{2^n\Gamma(1+\alpha)}{(ix+1)^{n+\alpha}}\sum_{l=0}^{n-1}(ix)^l\sum_{k=l}^{n-1}\frac{1}{2^k}\binom{n}{k}\binom{k}{l}\binom{-1-\alpha}{n-1-k}
	\cr
& = -\frac{2^n\Gamma(1+\alpha)}{(ix+1)^{n+\alpha}}\sum_{l=0}^{n-1}(ix)^{n-1-l}\sum_{k=0}^{l}\frac{1}{2^{n-1-k}}\binom{n}{n-1-k}\binom{n-1-k}{n-1-l}\binom{-1-\alpha}{k}
	\cr
& = -\frac{2n\Gamma(1+\alpha)}{(ix+1)^{n+\alpha}}\sum_{l=0}^{n-1}\binom{n-1}{l}(ix)^{n-1-l}\frac{1}{l+1}\sum_{k=0}^l\binom{l+1}{k+1}\binom{-1-\alpha}{k}2^k,
\end{align}

\noindent which yields \eqref{e:fraclaplnexpan}. The case with $n$ a negative integer follows from \eqref{e:lnpm}.

\end{proof}

\begin{remark}

If we only consider $l = 0$ in \eqref{e:fraclaplnexpan}, we get the asymptotic behavior of $(-\Delta)^{\alpha/2}\lambda_n(x)$:
\begin{equation*}
(-\Delta)^{\alpha/2}\lambda_n(x) \sim -\frac{2|n|\Gamma(1+\alpha)(i\sgn(n)x)^{|n|-1}}{(i\sgn(n)x+1)^{|n|+\alpha}} \sim -\frac{2|n|\Gamma(1+\alpha)}{(i\sgn(n)x)^{1+\alpha}}, \quad x\to\pm\infty,
\end{equation*}

\noindent i.e., $|(-\Delta)^{\alpha/2}\lambda_n(x)|$ decays at infinity as $|x|^{-1-\alpha}$ for all $n$, whereas
\begin{equation*}
\lambda_n(x) - 1 = \left(1 - \frac{2}{ix + 1}\right)^n - 1 \sim - \frac{2n}{ix + 1} \sim - \frac{2n}{ix}, \quad x\to\pm\infty,
\end{equation*}

\noindent i.e., $|\lambda_n(x) - 1|$ decays at infinity as $|x|^{-1}$, for all $n$. Therefore, the operator $(-\Delta)^{\alpha/2}$ introduces an extra decay of $|x|^{-\alpha}$.

\end{remark}

As a consequence of Lemma \ref{p:lemmaln}, we also have the following result.

\begin{proposition}

\label{c:corolCHSH}

The fractional Laplacian of \eqref{e:CH} and \eqref{e:SH} is given respectively by
\begin{align}
\label{e:fraclapCH}
(-\Delta)^{\alpha/2}\CH_{2n}(x) =
\begin{cases}
0, & n = 0,
	\\[1em]
\displaystyle{-\frac{2\Gamma(1+\alpha)}{(1+x^2)^{(1+\alpha)/2}}\sin\left((1+\alpha)\arccot(x)-\frac{\pi\alpha}2\right)}, & n = 1,
	\\[1em]
\displaystyle{-\frac{2n\Gamma(1+\alpha)}{(1+x^2)^{(n+\alpha)/2}}\sin\left((n+\alpha)\arccot(x)-\frac{\pi\alpha}2\right)}
	\\[1em]
\qquad\displaystyle{\times\sum_{l=0}^{\lfloor\frac{n-1}2\rfloor}(-1)^l\binom{n-1}{2l}x^{n-1-2l}{}_2F_1(-2l,1+\alpha; 2; 2)}
	\\[1em]
\quad\displaystyle{+\frac{2n\Gamma(1+\alpha)}{(1+x^2)^{(n+\alpha)/2}}\cos\left((n+\alpha)\arccot(x)-\frac{\pi\alpha}2\right)}
	\\[1em]
\qquad\displaystyle{\times\sum_{l=0}^{\lfloor\frac{n-2}2\rfloor}(-1)^l\binom{n-1}{2l+1} x^{n-2-2l}{}_2F_1(-2l-1,1+\alpha; 2; 2)}, & n = 2, 3, 4, \ldots,
\end{cases}
\end{align}

\noindent and

\begin{align}
\label{e:fraclapSH}
  (-\Delta)^{\alpha/2}\SH_{2n+1}(x) =
\begin{cases}
\displaystyle{\frac{2\Gamma(1+\alpha)}{(1+x^2)^{(1+\alpha)/2}}\cos\left((1+\alpha)\arccot(x)-\frac{\pi\alpha}2\right)}, & n  = 0,
\\[1em]
\displaystyle{\frac{2(n+1)\Gamma(1+\alpha)}{(1+x^2)^{(n+1+\alpha)/2}}\cos\left((n+1+\alpha)\arccot(x)-\frac{\pi\alpha}2\right)}
\\[1em]
\qquad\displaystyle{\times\sum_{l=0}^{\lfloor\frac{n}2\rfloor}(-1)^l\binom{n}{2l}x^{n-2l}{}_2F_1(-2l,1+\alpha; 2; 2)}
\\[1em]
\quad\displaystyle{+\frac{2n\Gamma(1+\alpha)}{(1+x^2)^{(n+1+\alpha)/2}}\sin\left((n+1+\alpha)\arccot(x)-\frac{\pi\alpha}2\right)}
\\[1em]
\qquad\displaystyle{\times\sum_{l=0}^{\lfloor\frac{n-1}2\rfloor}(-1)^l\binom{n}{2l+1} x^{n-1-2l}{}_2F_1(-2l-1,1+\alpha; 2; 2)}, & n = 1, 2, 3, \ldots.
\end{cases}
\end{align}
  
\end{proposition}

\begin{proof}
	
  From the definition of \eqref{e:CH} and \eqref{e:SH}, together with \eqref{e:lnpm} and Euler's formula $e^{is} = \cos(s) + i\sin(s)$, we have immediately, for $n=0$, $1$, $\ldots$,
$$
  \CH_{2n}(x)  = \Re(\lambda_n(x)),
  \quad \SH_{2n}(x)  = \Im(\lambda_{n+1}(x)),
$$
and this implies (see also \cite{boyd1990}) that
\begin{align*}
 \CH_{2n}(\cot(s)) = \cos(2ns), \quad \SH_{2n+1}(\cot(s)) = \sin(2(n+1)s).
\end{align*}

\noindent Since $\CH_{2n}(x)$ and $\SH_{2n}(x)$ are real, their fractional Laplacian is also real. We thus want to obtain expressions of $(-\Delta)^{\alpha/2}\CH_{2n}(x)$ and $(-\Delta)^{\alpha/2}\SH_{2n}(x)$ that avoid the use of complex numbers. To do this, we express all the complex numbers in \eqref{e:fraclaplnx1} in their polar form: e.g., $i^\alpha = e^{i\pi\alpha/2}$,
$$
\frac{1}{(x-i)^{n+\alpha}} =\frac{(x+i)^{n+\alpha}}{(1+x^2)^{n+\alpha}}=
\frac{e^{i(n+\alpha)\arccot(x)}}{(1+x^2)^{(n+\alpha)/2}},
$$

\noindent etc. Note that, since $x = \cot(s)$, with $s\in[0,\pi]$, we have chosen the definitions of the arctangent and the arccotangent such that $\arctan(1/x) = \arccot(x)\in[0, \pi]$ in the last expression. Then, \eqref{e:fraclaplnx1} becomes
\begin{equation}
\label{e:fraclaplnexp}
(-\Delta)^{\alpha/2}\lambda_n(x) = -\frac{2n\Gamma(1+\alpha)}{(1+x^2)^{(n+\alpha)/2}}\sum_{l=0}^{n-1}\binom{n-1}{l}x^{n-1-l}e^{i((n+\alpha)\arccot(x)-\pi(l+1+\alpha)/2)}{}_2F_1(-l,1+\alpha; 2; 2).
\end{equation}

\noindent With respect to \eqref{e:fraclapCH}, the case $n = 0$ is trivial, because $\CH_0(1) = 1$; otherwise, taking the real part of \eqref{e:fraclaplnexp},
\begin{align*}
(-\Delta)^{\alpha/2}\CH_{2n}(x) & = -\frac{2n\Gamma(1+\alpha)}{(1+x^2)^{(n+\alpha)/2}}\sum_{l=0}^{n-1}\binom{n-1}{l} x^{n-1-l}\sin\left(\left[(n+\alpha)\arccot(x)-\frac{\pi\alpha}2\right]-\frac{\pi l}2\right){}_2F_1(-l,1+\alpha; 2; 2)
	\cr
& = -\frac{2n\Gamma(1+\alpha)}{(1+x^2)^{(n+\alpha)/2}}\sin\left((n+\alpha)\arccot(x)-\frac{\pi\alpha}2\right)\sum_{l=0}^{n-1}\cos\left(\frac{\pi l}2\right)\binom{n-1}{l}x^{n-1-l}{}_2F_1(-l,1+\alpha; 2; 2)
\cr
& \phantom{{}={}} +\frac{2n\Gamma(1+\alpha)}{(1+x^2)^{(n+\alpha)/2}}\cos\left((n+\alpha)\arccot(x)-\frac{\pi\alpha}2\right)\sum_{l=0}^{n-1}\sin\left(\frac{\pi l}2\right)\binom{n-1}{l} x^{n-1-l}{}_2F_1(-l,1+\alpha; 2; 2).
\end{align*}

\noindent When $n = 1$, the first sum is equal to one, and the second one is equal to zero. For the other values of $n$, observe that $\cos(\pi l/2)$ is zero if and only if $l$ is odd, whereas $\sin(\pi l/2)$ is zero, if and only if $l$ is even. Therefore, substituting $l$ by $2l$ in the first term, and by $2l+1$ in the second term, the proof of \eqref{e:fraclapCH} is concluded.

Likewise, taking the imaginary part of \eqref{e:fraclaplnexp}, and replacing $n$ by $n + 1$:
\begin{align*}
(-\Delta)^{\alpha/2}\SH_{2n+1}(x) & = \frac{2(n+1)\Gamma(1+\alpha)}{(1+x^2)^{(n+1+\alpha)/2}}\sum_{l=0}^{n}\binom{n}{l} x^{n-l}\cos\left(\left[(n+1+\alpha)\arccot(x)-\frac{\pi\alpha}2\right]-\frac{\pi l}2\right){}_2F_1(-l,1+\alpha; 2; 2)
\cr
& = \frac{2(n+1)\Gamma(1+\alpha)}{(1+x^2)^{(n+1+\alpha)/2}}\cos\left((n+1+\alpha)\arccot(x)-\frac{\pi\alpha}2\right)\sum_{l=0}^{n}\cos\left(\frac{\pi l}2\right)\binom{n}{l}x^{n-l}{}_2F_1(-l,1+\alpha; 2; 2)
\cr
& \phantom{{}={}} +\frac{2(n+1)\Gamma(1+\alpha)}{(1+x^2)^{(n+1+\alpha)/2}}\sin\left((n+1+\alpha)\arccot(x)-\frac{\pi\alpha}2\right)\sum_{l=0}^{n}\sin\left(\frac{\pi l}2\right)\binom{n}{l} x^{n-l}{}_2F_1(-l,1+\alpha; 2; 2).
\end{align*}

\noindent When $n = 0$, the first sum is equal to one, and the second one is equal to zero. For the other values of $n$, we substitute again $l$ by $2l$ in the first term, and by $2l+1$ in the second term, to concluded the proof of \eqref{e:fraclapSH}.

\end{proof}

As before, in order to compute the fractional Laplacian of the cosine-like and sine-like Christov functions, we first express \eqref{e:fraclapmun} as a polynomial on $x$ multiplied by a rational function as follows:

\begin{lemma}
	
\label{p:lemmamun}
	
Let $\mu_n(x)$ be defined as in \eqref{e:mun}. Then, \eqref{e:fraclapmun} can be expanded as
\begin{equation}\label{e:fraclapmux}
(-\Delta)^{\alpha/2}\mu_n(x) =
\begin{cases}
\displaystyle{\frac{\Gamma(1+\alpha)}{(ix+1)^{n+1+\alpha}}\sum_{l = 0}^n\binom{n}{l}(ix)^{n-l}{}_2F_1(-l, 1 + \alpha; 1; 2)}, & n = 0, 1, \ldots,
	\\[1em]
\displaystyle{-\frac{\Gamma(1+\alpha)}{(-ix+1)^{-n+\alpha}}\sum_{l = 0}^{-1-n}\binom{-1-n}{l}(-ix)^{-1-n-l}{}_2F_1(-l, 1 + \alpha; 1; 2)}, & n = -1, -2, \ldots.	
\end{cases}
\end{equation}

\end{lemma}

\begin{proof}

Assume $n$ is a nonnegative number. The proof is almost identical to that of Proposition \ref{p:lemmaln}. Indeed, starting from the second last line in \eqref{e:fraclapmun} and following the same steps as in \eqref{e:fraclaplnx0} and \eqref{e:fraclaplnx1},
\begin{align*}
(-\Delta)^{\alpha/2}\mu_n(x) & = \frac{2^n\Gamma(1+\alpha)}{(ix+1)^{n+1+\alpha}}\sum_{k = 0}^n\sum_{l=0}^{k}\frac{1}{2^k}\binom{n}{k}\binom{k}{l}\binom{-1-\alpha}{n-k}(ix)^{l}
	\cr
& = \frac{\Gamma(1+\alpha)}{(ix+1)^{n+1+\alpha}}\sum_{l = 0}^n\binom{n}{l}(ix)^{n-l}\sum_{k=0}^{l}\binom{l}{k}\binom{-1-\alpha}{k}2^k,
\end{align*}

\noindent which is \eqref{e:fraclapmux}. The case with $n$ a negative integer follows from \eqref{e:munpm}.

\end{proof}

\begin{remark}
	
As we did for $(-\Delta)^{\alpha/2}\lambda_n(x)$, if we only consider $l = 0$ in \eqref{e:fraclapmux}, we get the asymptotic behavior of $(-\Delta)^{\alpha/2}\mu_n(x)$, as $x\to\pm\infty$:
\begin{equation*}
(-\Delta)^{\alpha/2}\mu_n(x) \sim
\begin{cases}
\dfrac{\Gamma(1+\alpha)(ix)^{n}}{(ix+1)^{n+1+\alpha}} \sim \dfrac{\Gamma(1+\alpha)}{(ix)^{1+\alpha}}, & n = 0, 1, \ldots,
\\[1em]
-\dfrac{\Gamma(1+\alpha)(-ix)^{-1-n}}{(-ix+1)^{-n+\alpha}} \sim -\dfrac{\Gamma(1+\alpha)}{(-ix)^{1+\alpha}}, & n = -1, -2, \ldots;
\end{cases}
\end{equation*}

\noindent whereas
\begin{equation*}
\mu_n(x) = \frac12\left[\left(1 - \frac{2}{ix + 1}\right)^n - \left(1 - \frac{2}{ix + 1}\right)^{n+1}\right]\sim \frac{1}{ix + 1} \sim \frac{1}{ix}, \quad x\to\pm\infty.
\end{equation*}

\noindent Therefore, as expected, the operator $(-\Delta)^{\alpha/2}$ introduces an extra decay of $|x|^{-\alpha}$.

\end{remark}

On the other hand, applying now Lemma \ref{p:lemmamun}, we have the following result:

\begin{proposition}

The fractional Laplacian of \eqref{e:CC} and \eqref{e:SC} is given respectively by

\begin{align}
\label{e:fraclapCC}
(-\Delta)^{\alpha/2}\CC_{2n}(x) & = \frac{2n\Gamma(1+\alpha)}{(1+x^2)^{(n+1+\alpha)/2}}\sin\left((n+1+\alpha)\arccot(x)-\frac{\pi\alpha}2\right)
\cr
& \qquad \times\sum_{l=0}^{\lfloor\frac{n}2\rfloor}(-1)^l\binom{n}{2l}x^{n-2l}{}_2F_1(-2l,1+\alpha; 2; 1)
\cr
& \quad -\frac{2n\Gamma(1+\alpha)}{(1+x^2)^{(n+1+\alpha)/2}}\cos\left((n+1+\alpha)\arccot(x)-\frac{\pi\alpha}2\right)
\cr
& \qquad \times\sum_{l=0}^{\lfloor\frac{n-1}2\rfloor}(-1)^l\binom{n}{2l+1} x^{n-2l-1}{}_2F_1(-2l-1,1+\alpha; 2; 1),
\end{align}

\noindent and
\begin{align}
\label{e:fraclapSC}
(-\Delta)^{\alpha/2}\SC_{2n+1}(x) & = \frac{2n\Gamma(1+\alpha)}{(1+x^2)^{(n+1+\alpha)/2}}\cos\left((n+1+\alpha)\arccot(x)-\frac{\pi\alpha}2\right) 
\cr
& \qquad \times \sum_{l=0}^{\lfloor\frac{n}2\rfloor}(-1)^l\binom{n}{2l}x^{n-2l}{}_2F_1(-2l,1+\alpha; 2; 1)
\cr
& \quad +\frac{2n\Gamma(1+\alpha)}{(1+x^2)^{(n+1+\alpha)/2}}\sin\left((n+1+\alpha)\arccot(x)-\frac{\pi\alpha}2\right)
\cr
& \qquad \times
\sum_{l=0}^{\lfloor\frac{n-1}2\rfloor}(-1)^l\binom{n}{2l+1} x^{n-2l-1}{}_2F_1(-2l-1,1+\alpha; 2; 1).
\end{align}
	
\end{proposition}

\begin{proof}
	
The proof is very similar to that of Proposition \ref{c:corolCHSH}, so we omit some details. We first observe that, from \eqref{e:CC}, \eqref{e:SC} and \eqref{e:munpm} (see also \cite{boyd1990}),
$$
\CC_{2n}(x)  = \Re(\mu_n(x)), \quad
\SC_{2n+1}(x)  = -\Im(\mu_{n+1}(x)), \quad n = 0, 1, 2, \ldots;
$$

\noindent hence,
$$
\CC_{2n}(\cot(s)) = \frac{\cos(2ns) - \cos(2(n+1)s)}2, \quad
\SC_{2n+1}(\cot(s)) = \frac{\sin(2(n+1)s) - \sin(2ns)}2, \quad n = 0, 1, 2, \ldots.
$$
        
	\noindent We treat first the case where $n$ is a nonnegative integer. Then, we express all the complex numbers appearing in \eqref{e:fraclapmux} in their polar form, to obtain:
	\begin{equation}
	\label{e:fraclapmuexp}
	(-\Delta)^{\alpha/2}\mu_n(x) = \frac{\Gamma(1+\alpha)}{(1+x^2)^{(n+1+\alpha)/2}}\sum_{l = 0}^n\binom{n}{l}x^{n-l}e^{i((n+1+\alpha)\arccot(x) - \pi(l+1+\alpha)/2)}{}_2F_1(-l, 1 + \alpha; 1; 2).
	\end{equation}

	\noindent Now, taking the real part of \eqref{e:fraclapmuexp} gives
	\begin{align*}
	(-\Delta)^{\alpha/2}\CC_{2n}(x) & = \frac{\Gamma(1+\alpha)}{(1+x^2)^{(n+1+\alpha)/2}}\sum_{l=0}^{n}\binom{n}{l} x^{n-l}\sin\left(\left[(n+1+\alpha)\cot(x)-\frac{\pi\alpha}2\right]-\frac{\pi l}2\right){}_2F_1(-l,1+\alpha; 2; 1)
	\cr
	& = \frac{2n\Gamma(1+\alpha)}{(1+x^2)^{(n+1+\alpha)/2}}\sin\left((n+1+\alpha)\arccot(x)-\frac{\pi\alpha}2\right)\sum_{l=0}^{n}\cos\left(\frac{\pi l}2\right)\binom{n}{l}x^{n-l}{}_2F_1(-l,1+\alpha; 2; 1)
	\cr
	& \phantom{{}={}} -\frac{2n\Gamma(1+\alpha)}{(1+x^2)^{(n+1+\alpha)/2}}\cos\left((n+1+\alpha)\arccot(x)-\frac{\pi\alpha}2\right)\sum_{l=0}^{n}\sin\left(\frac{\pi l}2\right)\binom{n}{l} x^{n-l}{}_2F_1(-l,1+\alpha; 2; 1).
	\end{align*}
	
	\noindent Substituting $l$ by $2l$ in the first term, and by $2l+1$ in the second term of the last expression, we get \eqref{e:fraclapCC}.

    Likewise, taking the imaginary part of \eqref{e:fraclaplnexp} and changing the sign, we have
	\begin{align*}
	(-\Delta)^{\alpha/2}\SC_{2n+1}(x) & = \frac{\Gamma(1+\alpha)}{(1+x^2)^{(n+1+\alpha)/2}}\sum_{l=0}^{n}\binom{n}{l} x^{n-l}\cos\left(\left[(n+1+\alpha)\cot(x)-\frac{\pi\alpha}2\right]-\frac{\pi l}2\right){}_2F_1(-l,1+\alpha; 2; 1)
	\cr
	& = \frac{2n\Gamma(1+\alpha)}{(1+x^2)^{(n+1+\alpha)/2}}\cos\left((n+1+\alpha)\arccot(x)-\frac{\pi\alpha}2\right) \sum_{l=0}^{n}\cos\left(\frac{\pi l}2\right)\binom{n}{l}x^{n-l}{}_2F_1(-l,1+\alpha; 2; 1)
	\cr
	& \phantom{{}={}} +\frac{2n\Gamma(1+\alpha)}{(1+x^2)^{(n+1+\alpha)/2}}\sin\left((n+1+\alpha)\arccot(x)-\frac{\pi\alpha}2\right)\sum_{l=0}^{n}\sin\left(\frac{\pi l}2\right)\binom{n}{l} x^{n-l}{}_2F_1(-l,1+\alpha; 2; 1).
	\end{align*}
	
	\noindent Finally, substituting $l$ by $2l$ in the first term, and by $2l+1$ in the second term of the last expression, we get \eqref{e:fraclapSC}.
\end{proof}

\subsection{Cases with $\alpha\in\{0, 1, 2\}$}

Since we are working with $\alpha\in(0,2)$, it is interesting to see what happens if we evaluate \eqref{e:fraclapln}, \eqref{e:fraclapmun} and their cosine-like and sine-like expressions at the ends and at the middle point of the interval. For that purpose, we first express $\lambda_n(x)$ in \eqref{e:ln} as
$$
\lambda_n(x) = \left(\frac{i\sgn(n)x - 1}{i\sgn(n)x + 1}\right)^{|n|};
$$

\noindent hence,
\begin{align*}
\lambda_n'(x) & = \frac{2in}{(i\sgn(n)x + 1)^2}\left(\frac{i\sgn(n)x - 1}{i\sgn(n)x + 1}\right)^{|n|-1},
\cr
\lambda_n''(x) & = \frac{4inx - 4n^2}{(i\sgn(n)x + 1)^4}\left(\frac{i\sgn(n)x - 1}{i\sgn(n)x + 1}\right)^{|n|-2}.
\end{align*}

\noindent We also need Newton's binomial formula and its derivative:
$$
(1+z)^{n-1} = \sum_{k=0}^{n-1}\binom{n-1}{k}z^k \Longrightarrow (n-1)(1+z)^{n-2} = \sum_{k=1}^{n-1}k\binom{n-1}{k}z^{k-1}.
$$


\noindent Let us consider first the case with $\alpha = 0$, $n\in\mathbb Z\backslash\{0\}$:
\begin{align*}
\lim_{\alpha\to0^+}(-\Delta)^{\alpha/2}\lambda_n(x) & =
-\frac{2|n|\Gamma(1)}{i\sgn(n)x+1}{}_2F_1\left(1-|n|,1; 2;\frac{2}{i\sgn(n)x+1}\right)
\cr
& = \sum_{k=0}^{|n|-1}\binom{|n|}{k+1}\left(\frac{-2}{i\sgn(n)x+1}\right)^{k+1} = \sum_{k=1}^{|n|}\binom{|n|}{k}\left(\frac{-2}{i\sgn(n)x+1}\right)^k
\cr
& = \left(1 - \frac{2}{i\sgn(n)x+1}\right)^{|n|} - 1  = \lambda_n(x) - 1.
\end{align*}

\noindent Therefore, for $n = 0, 1, 2, \ldots$,
\begin{align*}
\lim_{\alpha\to0^+}(-\Delta)^{\alpha/2}\CH_{2n}(x) & = \CH_{2n}(x) - 1,
	\cr
\lim_{\alpha\to0^+}(-\Delta)^{\alpha/2}\SH_{2n+1}(x) & = \SH_{2n+1}(x).
\end{align*}

\noindent Observe that, as with $\lambda_0(x) = 1$ in \eqref{e:discontinuousa0}, this limit is singular for $\lambda_n(x)$ and $\CH_{2n}$, because their respective Fourier transforms are not classical functions. On the other hand, the limit for $\SH_{2n+1}(x)$ is continuous, because the function belongs to $\mathcal L^2(\mathbb R)$, and, hence, its Fourier transform is well defined. Therefore, if we consider instead $\lambda_n(x) - 1$ and $\CH_{2n} - 1$, the functions belong now to $\mathcal L^2(\mathbb R)$, and 
\begin{align*}
\lim_{\alpha\to0^+}(-\Delta)^\alpha(\lambda_n(x) - 1) = \lambda_n(x) - 1 = (-\Delta)^0(\lambda_n(x) - 1),
	\\
\lim_{\alpha\to0^+}(-\Delta)^\alpha(\CH_{2n}(x) - 1) = \CH_{2n}(x) - 1 = (-\Delta)^0(\CH_{2n}(x) - 1),
\end{align*}

\noindent i.e., the limit is now continuous. Likewise, the limit is also continuous for $\mu_n(x)$ (and, hence, for $\CC_{2n}(x)$ and $\SC_{2n+1}(x)$), because they are in $\mathcal L^2(\mathbb R)$. Indeed, from \eqref{e:mulambda},
$$
\lim_{\alpha\to0^+}(-\Delta)^{\alpha/2}\mu_n(x) = \lim_{\alpha\to0^+} \frac{(-\Delta)^{\alpha/2}\lambda_{n}(x) - (-\Delta)^{\alpha/2}\lambda_{n+1}(x)}2 = \frac{\lambda_n(x) - \lambda_{n+1}(x)}2 = (-\Delta)^0 \mu_n(x).
$$

\noindent Unlike the previous case, the case with $\alpha = 2$ poses no problems:
\begin{align*}
\lim_{\alpha\to2}(-\Delta)^{\alpha/2}\lambda_n(x) & =
-\frac{2|n|\Gamma(3)}{(i\sgn(n)x+1)^{3}}{}_2F_1\left(1-|n|,3; 2;\frac{2}{i\sgn(n)x+1}\right)
\cr
& = -\frac{2|n|}{(i\sgn(n)x+1)^{3}}\sum_{k=0}^{|n|-1}\binom{|n|-1}{k}(k+2)\left(-\frac{2}{i\sgn(n)x+1}\right)^{k}
\cr
& = \frac{4|n|(|n|-1)}{(i\sgn(n)x+1)^{4}}\left(1-\frac{2}{i\sgn(n)x+1}\right)^{|n|-2} - \frac{4|n|}{(i\sgn(n)x+1)^{3}}\left(1-\frac{2}{i\sgn(n)x+1}\right)^{|n|-1}
\cr
& = -\lambda_n''(x),
\end{align*}

\noindent i.e., we recover $(-\Delta)^{1}\lambda_n(x)$. This is also true for all $\mu_n(x)$, $\CH_{2n}(x)$, $\SH_{2n+1}(x)$, $\CC_{2n}(x)$ and $\SC_{2n+1}(x)$.

Finally, when $\alpha = 1$,
\begin{align*}
\lim_{\alpha\to1}(-\Delta)^{\alpha/2}\lambda_n(x) & =
-\frac{2|n|\Gamma(2)}{(i\sgn(n)x+1)^{2}}{}_2F_1\left(1-|n|,2; 2;\frac{2}{i\sgn(n)x+1}\right)
\cr
& = -\frac{2|n|}{(i\sgn(n)x+1)^{2}}\sum_{k=0}^{|n|-1}\binom{|n|-1}{k}\left(-\frac{2}{i\sgn(n)x+1}\right)^{k}
\cr
& = -\frac{2|n|}{(i\sgn(n)x+1)^{2}}\left(1 - \frac{2}{i\sgn(n)x+1}\right)^{|n|-1} = i\sgn(n)\lambda_n'(x).
\end{align*}

\noindent Note that, if we express the last equality in terms of $s$, we obtain $2|n|\sin^2(s)e^{i2ns}$, which is precisely the expression for $(-\Delta)^{1/2}\lambda_n(\cot(s))$ obtained in \cite{cayamacuestadelahoz2019} using a completely different approach. Consequently, for $n = 0, 1, 2, \ldots$, from \eqref{e:CH} and \eqref{e:SH}, respectively,
\begin{align*}
(-\Delta)^{1/2}\CH_{2n}(x) & = \frac{i\sgn(n)\lambda_n'(x) + i\sgn(-n)\lambda_{-n}'(x)}{2} =
\begin{cases}
0, & n = 0,
\cr
-\SH_{2n-1}'(x), & n\in\mathbb N,
\end{cases}
	\\
(-\Delta)^{1/2}\SH_{2n+1}(x) & = \frac{i\sgn(n+1)\lambda_{n+1}'(x) - i\sgn(-n-1)\lambda_{-n-1}'(x)}{2i} = \CH_{2n+2}'(x).
\end{align*}

\noindent Similarly, in the case of $\mu_n(x)$,
$$
(-\Delta)^{1/2}\mu_n(x) = \frac{i\sgn(n)\lambda_n'(x) - i\sgn(n+1)\lambda_{n+1}'(x)}{2} 
=
\begin{cases}
i\mu_n'(x), & n = 0, 1, 2, \ldots,
	\\
-i\mu_n'(x), & n = -1, -2, -3, \ldots,
\end{cases}
$$

\noindent or, equivalently, $(-\Delta)^{1/2}\mu_n(x) = i\sgn(n+1/2)\mu_n'(x)$. Therefore, for $n = 0, 1, 2, \ldots$, from \eqref{e:CC} and \eqref{e:SC}, respectively,
\begin{align*}
(-\Delta)^{1/2}\CC_{2n}(x) & = \frac{i\sgn(n+1/2)\mu_n'(x) - i\sgn(-n-1/2)\mu_{-n-1}'(x)}{2} = \SC_{2n+1}'(x),
\\
(-\Delta)^{1/2}\SC_{2n+1}(x) & = -\frac{i\sgn(n+1/2)\mu_n'(x) + i\sgn(-n-1/2)\mu_{-n-1}'(x)}{2i} = -\CC_{2n}'(x).
\end{align*}

\noindent Observe that all the equalities for $\alpha = 1$ can be derived from the fact that $\lambda_n(x)$ and $\mu_n(x)$ are eigenfunctions of the Hilbert transform (see for instance \cite{boydxu2011,weideman1995}).

\section{Numerical implementation of \eqref{e:fraclapln}}

\label{s:numerical}

In this Section, we will focus on the numerical implementation of \eqref{e:fraclapln}, because the other functions considered in this paper can be expressed in terms of $\lambda_n(x)$. Moreover, we will reduced ourselves to the case with $n\in\mathbb N$, because the case with nonpositive $n$ is trivially reduced to the former. All the experiments that follow have been carried out in an HP ZBook 15 G3 with processor Intel(R) Core(TM) i7-6700HQ CPU @ 2.60GHz, graphic card NVIDIA Quadro M1000M, and 16384MB of RAM, under Windows 7 Enterprise with Service Pack 1. We have mainly worked with \textsc{MATLAB} R2019b \cite{matlab}, but have also used \textsc{Mathematica} 11.3 \cite{mathematica} in a few examples.

Let us start by discussing the drawbacks of attempting to implement \eqref{e:fraclapln} by using implementations of ${}_2F_1$ in commercial mathematical software. Even if commands that approximate numerically ${}_2F_1$ functions are available in, among others, \textsc{Matlab} and \textsc{Mathematica}, it seems that there is still no satisfactory implementation of ${}_2F_1$ for all its possible values. For example, in our case, there seem to be cancellation errors, due to the fact that we are adding and subtracting alternatively very large numbers. Recall that the infinite series in \eqref{e:2F1} converges in general for $|z| < 1$, which in our case means $|z| = |2/(ix+1)| < 1$ or $|x| > \sqrt 3$, even if we are interested in evaluating it at any $x\in\mathbb R$. On the other hand, since the first parameter of ${}_2F_1$ in \eqref{e:fraclapln}, $1-n$, is negative or zero, it implies that the numerical approximation of \eqref{e:fraclapln} is reduced to computing a finite sum, which is always convergent, so, in principle, we should expect no serious accuracy issues from a commercial implementation. However, if we consider for instance $x = 0$, for which $|2/(ix+1)|$ is largest (and, hence, the worst case), it is easy to check that both \textsc{Matlab} and \textsc{Mathematica} return wrong results for not too large values of $n$. Take for example $n = 400$ and $\alpha = 1.5$; then, the \textsc{Mathematica} command \texttt{Hypergeometric2F1[-399, 2.5, 2, 2]} gives $6.2771\cdot10^{57}$, whereas the \textsc{Matlab} command \textsc{hypergeom([-399, 2.5], 2, 2)} yields $4.9217\cdot10^{18}$. This is due to the fact that the signs of $(1-n)_k$ alternate, whereas $(1+\alpha)_k$ is always positive; hence, we are summing large quantities with alternating signs, for which a 64-bit floating point representation is clearly insufficient to store all the required digits, giving rise to severe rounding errors. Nevertheless, it is possible to obtain the correct results with both programs by just introducing minor modifications. More precisely, both \texttt{N[Hypergeometric2F1[-399, 5/2, 2, 2]]} and \texttt{double(hypergeom([-399, str2sym('5/2')], 2, 2))} return the correct result, ${}_2F_1(-399, 2.5; 2; 2) = -21.2769\ldots$. In fact, if we omit respectively the commands \texttt{N} and \texttt{double}, we get a fraction with very large nominator and denominator. Observe that, in \textsc{Mathematica}, \texttt{5/2} is treated natively in a different way as \texttt{2.5}, which is stored using a floating point representation; whereas, in \textsc{Matlab}, to obtain a similar effect, the Symbolic Math Toolbox needs to be installed. Then, \texttt{str2sym('5/2')} transforms the string \texttt{'5/2'} into a symbolic object that stores the fraction $5/2$.

In the rest of the section, for the sake of simplicity, we will work exclusively with \textsc{Matlab}. The previous arguments suggest working with fractions whenever possible. However, if some of the parameters of ${}_2F_1$ are irrational, approximating them adequately by fractions might involve large numerators and denominators, which does not seem ideal, either. Instead, we have found more advisable to work with a larger number of digits, which can be done by means of the function \texttt{vpa} included in the Symbolic Math Toolbox. The name of this command comes from the initials of ``variable-precision arithmetic'', i.e., arbitrary-precision arithmetic, and evaluates by default its argument to at least 32 digits, even if another number of digits can be specified as a second argument of \texttt{vpa}; e.g., \texttt{vpa(str2sym('pi'), 500)} returns $500$ exact digits of $\pi$. Furthermore, it is possible to change globally the default number of digits to $d$ by means of \texttt{digits(d)}. Take for instance $n = 400$, $\alpha = \sqrt 3$, then \texttt{hypergeom([-399, 1 + sqrt(3)], 2, 2)} returns $-1.4831\cdot10^{18}$; whereas \texttt{digits(24);} followed by \texttt{hypergeom([-399, vpa(str2sym('1 + sqrt(3)'))], 2, 2)} returns the exact value of the first 24 digits of ${}_2F_1(-399, 1 + \sqrt 3; 2; 2)$, namely $-84.1742043148953740804326$. Therefore, in our opinion, the use of arbitrary precision is the easiest and safest way to guarantee that the results returned by ${}_2F_1$ are correct, and, indeed, research is currently been conducted on this area (see for instance \cite{johansson} and its references, where ${}_2F_1$ and other hypergeometric functions are computed in arbitrary-precision interval arithmetic).

Coming back to the implementation of \eqref{e:fraclapln}, note that \texttt{z} in \texttt{hypergeom([a,b],c,z)} can be a vector, but \texttt{a}, \texttt{b} and \texttt{c} must be scalars. Therefore, if we have a function $u(x)$ represented as say
$$
u(x) = \sum_{n=-N}^{N}a_n\lambda_{n}(x),
$$

\noindent and want to approximate its fractional Laplacian by using \texttt{hypergeom}, we have to invoke this command $N$ times using arbitrary precision, i.e. for $n = 1, \ldots, N$, which can be extremely expensive from a computational point of view. For instance, executing \texttt{hypergeom([1-n, vpa(str2sym('1 + sqrt(3)'))], 2, 2)}, for $n = 1, 2, \ldots, 40$, using $24$ digits, requires $24.31$ seconds, even if we are considering just one single value of $z$. To avoid this, we compute \eqref{e:fraclapln} by using directly a series representation, but in such a way that the evaluations for different $n$ are done simultaneously (and as much independently from $\alpha$ as possible), which is much more efficient. More precisely, we express \eqref{e:fraclapln} in the following form and evaluate it at $M$ different spatial nodes $x_0, \ldots, x_{M-1}$:
\begin{equation}
\label{e:fraclaplnpossum}
(-\Delta)^{\alpha/2}\lambda_n(x_j) = -\Gamma(1+\alpha)(ix+1)^{-\alpha}\sum_{k=1}^n\binom{n}{k}\binom{-1-\alpha}{k-1}\left(\frac{2}{ix_j+1}\right)^k, \quad n\in\{1, \ldots, N\}.
\end{equation}

\noindent After converting $\alpha$ and $\pi$ to arbitrary precision, we generate the spatial nodes and store them in one column vector $\mathbf x = (x_0, \ldots, x_{M-1})^T$. Then, we create a matrix $\mathbf A\in\mathcal M_{M\times N}(\mathbb C)$, such that its columns are precisely the powers $\{1, \ldots, N\}$ of the entries of $(2 / (ix_0 + 1), \ldots, 2 / (ix_{M-1} + 1))^T$:
$$
\mathbf A =
\begin{pmatrix}
\left(\dfrac{2}{ix_0 + 1}\right)^1 & \cdots & \left(\dfrac{2}{ix_0 + 1}\right)^N
	\\
\vdots & \ddots & \vdots
	\\
\left(\dfrac{2}{ix_{M-1} + 1}\right)^1 & \cdots & \left(\dfrac{2}{ix_{M-1} + 1}\right)^N
\end{pmatrix}.
$$

\noindent This can be done in parallel by simply typing \texttt{A = (2 ./ (1i * x + 1)) .\^{} (1:N);}. Afterward, we create a matrix $\mathbf B\in\mathcal M_{N\times N}(\mathbb C)$ that stores the binomial coefficients computed using arbitrary precision:
$$
\mathbf B =
\begin{pmatrix}
\binom{1}{1} & \binom{2}{1} & \binom{3}{1} & \cdots & \binom{N}{1}
	\\[0.5em]
0 & \binom{2}{2} & \binom{3}{2} & \cdots & \binom{N}{2}
	\\[0.5em]
0 & 0 & \binom{3}{3} & \cdots & \binom{N}{3}
	\\
\vdots & \vdots & \vdots & \ddots & \vdots
	\\
0 & 0 & 0 & \cdots & \binom{N}{N}
\end{pmatrix},
$$

\noindent which can be done column-wise by using that $\binom{n+1}{k+1} = \binom{n}{k+1} + \binom{n}{k}$. However, if we apply recursively this property, we get
$$
\binom{n+1}{k+1} = \binom{n}{k} + \binom{n}{k+1}  = \binom{n}{k} + \binom{n-1}{k} + \binom{n-1}{k+1} = \ldots = \sum_{j=0}^{n}\binom{j}{k},
$$

\noindent which allows to generate $\mathbf B$ row-wise as well.

This alternative implementation appears to be faster and is the one we have adopted. More precisely, after initializing $\mathbf B$ with \texttt{B = vpa(zeros(N, N));}, we store the first row, which is explicitly known, i.e., \texttt{B(1, :) = 1:N;}. Then, each remaining row is generated recursively by computing the cumulative sum of the entries of the previous row, i.e., \texttt{for n = 2:N, B(n, n:N) = cumsum(B(n-1, n-1:N-1)); end}. At this point, it is important to underline that both $\mathbf A$ and $\mathbf B$ are completely independent from $\alpha$, so they can be generated once and then be reused to evaluate \eqref{e:fraclaplnpossum} at different values of $\alpha$. Furthermore, the parts of \eqref{e:fraclaplnpossum} that depend actually on $\alpha$ (which we store in an arbitrary-precision variable \texttt{a}) can be stored in just two vectors $\mathbf c\in\mathcal M_{M\times1}(\mathbb C)$ and $\mathbf d\in\mathcal M_{1\times N}(\mathbb C)$; the former is a column vector consisting of the entries of $i\mathbf x + 1$ raised to the power $-\alpha$, i.e., \texttt{c = (1i * x + 1) .\^{} (-a);}, whereas the latter is a row vector formed by the values of $-\Gamma(1+\alpha)\binom{-1-\alpha}{k-1}$, for $k = 0, \ldots N-1$. Moreover, we observe that the entries of $\mathbf d$ can be generated by obtaining the cumulative product of the following quantities:
$$
\left\{-\Gamma(1+\alpha), \frac{-1-\alpha}{1}, \frac{-2-\alpha}{2}, \ldots, \frac{-(N+1)-\alpha}{N-1}\right\},
$$

\noindent i.e., \texttt{d = cumprod([-gamma(1 + a), -1 - a ./ (1:N-1)]);}. Combining all the previously defined elements, we generate a matrix $\M_\alpha\in\mathcal M_{M\times N}(\mathbb C)$, such that its columns correspond to \eqref{e:fraclaplnpossum} evaluated at $n = 1, \ldots, N$, respectively:
$$
\M_\alpha = \diag(\mathbf c)\mathbf A \diag(\mathbf d) \mathbf B,
$$

\noindent where $\diag(\mathbf c)\in\mathcal M_{M\times M}(\mathbb C)$ and $\diag(\mathbf d)\in\mathcal M_{N\times N}(\mathbb C)$ are the diagonal matrices whose diagonal entries are $\mathbf c$ and $\mathbf d$ respectively; in \textsc{Matlab}, \texttt{Ma = diag(c) * A * diag(d) * B;}. In this regard, let us mention that we have found that this equivalent expression executes a bit faster:
$$
\M_\alpha = (\mathbf A \circ (\mathbf c\cdot\mathbf d)) \mathbf B,
$$

\noindent where $\circ$ denotes the Hadamard product, i.e., the entrywise product between matrices or equal size. In that case, \texttt{Ma = (A .* (c * d)) * B;}. Finally, the resulting matrix can be cast to the 64-bit floating point format by means of \texttt{Ma = double(Ma);}. Observe that, if $\M_\alpha$ is very large, it might happen that \texttt{double} cannot be applied to a whole arbitrary-precision matrix, because the capacity of \textsc{Matlab} is exceeded; in that case, the conversion has to be done row-wise or column-wise, although the global impact is small.

In the following subsection, we will perform the numerical experiments; special attention will be given to the number of digits necessary to compute $\M_\alpha$ with satisfactory accuracy.

\subsection{Numerical experiments}

\label{s:numericalexperiments}

The procedure to generate $\M_\alpha\in\mathcal M_{M\times N}(\mathbb C)$ that we have just explained is by definition \emph{exact}, provided that a large enough number of digits is chosen. In our numerical experiments, we have observed that the number of necessary digits is virtually independent from the choice of $\alpha\in(0,2)$; in what follows, we have used systematically the irrational value $\alpha = \sqrt 3$. We have first considered only the worst case $x = 0$, and have thus approximated $(-\Delta)^{\sqrt3/2}\lambda_N(0)$, i.e., have computed the last element of the row vector $\M_{\sqrt3}\in\mathcal M_{1\times N}(\mathbb C)$ using arbitrary precision with an increasing number of digits, until satisfactory convergence is achieved. Let $m_N^{(k)}$ denote the $k$-digit approximation of $(-\Delta)^{\sqrt3/2}\lambda_N(0)$; then, we take as stopping criterion that
\begin{equation}
\label{e:epsmNk}
\left|\frac{m_N^{(k)} - m_N^{(k+1)}}{m_N^{(k+1)}}\right| < \varepsilon,
\end{equation}

\noindent where $\varepsilon = 2^{-52}$, \texttt{eps} in \textsc{Matlab}, is the so-called machine epsilon in 64-bit floating point arithmetic. Note that the use of the relative discrepancy between $m_N^{(k)}$ and $m_N^{(k+1)}$ is justified because $|(-\Delta)^{\sqrt3/2}\lambda_N(0)|$ quickly grows with $N$ (we have found experimentally that $|(-\Delta)^{\sqrt3/2}\lambda_N(0)| \approx 3.322 N^{\sqrt 3}$), and we are interested in choosing the adequate number of global digits (before and after the decimal point), not of correct decimals. Remark that, in the numerical experiments, we have found that, in general, the number of digits required does not diminish with $N$, even if, very occasionally, that quantity can be slightly smaller for $N+1$ than for $N$. Hence, in order to determine the number of digits for a given $N + 1$, it is completely safe to start generating $m_{N+1}^{(k)}$ with the same number of digits $k$ necessary to generate $m_{N}^{(k)}$ satisfying \eqref{e:epsmNk}, and then increase $k$ by one unity, until the convergence condition \eqref{e:epsmNk} is fulfilled. In Figure \ref{f:digitsMM0}, we have plotted the number of digits with respect to $N$; note that the graph is roughly a straight line with slope slightly less than $1/2$; in fact, the least squares regression line is $\hat y = 0.47598x + 6.6938$.

\begin{figure}[!htbp]
	\centering
	\includegraphics[width=0.5\textwidth, clip=true]{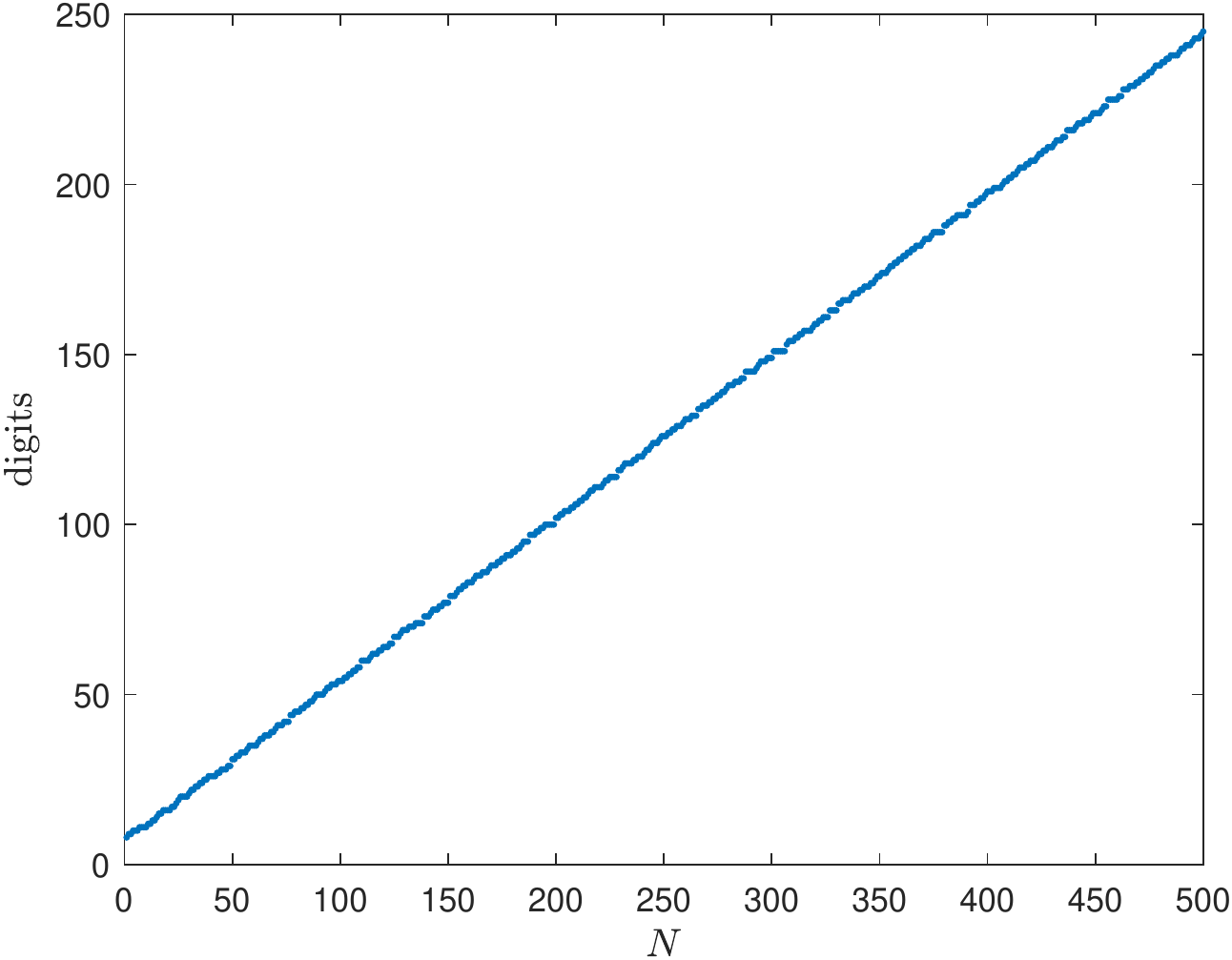}
	\caption{Number of digits necessary to approximate $\M_\alpha\in\mathcal M_{1\times N}(\mathbb C)$ with accuracy \eqref{e:epsmNk}, as a function of $N$.}
	\label{f:digitsMM0}
\end{figure}

On the other hand, for a given number of columns $N$, we have also considered $N$ nodes $\{x_0, \ldots, x_{N-1}\}$ and approximated the whole matrix $\M_{\sqrt3}\in\mathcal M_{N\times N}(\mathbb C)$. In order to compare the results with those in \cite{cayamacuestadelahoz2019}, we have taken equally-spaced non-end nodes $s_j$, which determine the choice of $x_j$:
\begin{equation}
\label{e:nodessj}
s_j = \frac{\pi(2j+1)}{2N} \Longrightarrow x_j = \cot\left(\frac{\pi(2j+1)}{2N}\right), \quad 0\le j\le N-1.
\end{equation}

\noindent Let $\M_{\sqrt3}^{(k)} \in\mathcal M_{N\times N}(\mathbb C)$ denote the approximation of $(-\Delta)^{(\sqrt3/2)}\lambda_n(x_j)$, for $n\in\{1, \ldots, N\}$, $j\in\{0, \ldots, N-1\}$, using arbitrary precision arithmetic with $k$ digits. Bearing in mind that $x_{N-1-j} = -x_j$, we have $(-\Delta)^{(\sqrt3/2)}\lambda_n(x_{N-1-j}) = (-\Delta)^{(\sqrt3/2)}\lambda_n(-x_j) = \overline{(-\Delta)^{(\sqrt3/2)}\lambda_n(x_j)}$, so the last $\lfloor N / 2\rfloor$ rows of $\M_{\sqrt3}^{(k)}$ can be generated from its first $\lfloor N / 2\rfloor$ rows, which reduces the global computational cost. Observe that, from \eqref{e:fraclaplnexpan}, it follows that $\lim_{x\to \pm\infty}(-\Delta)^{\sqrt3/2}|\lambda_n(x)| = 0$ in \eqref{e:fraclapln}, for all $n$, so, in general, we can expect very large differences in the orders of magnitude of the different entries of $\M_\alpha$. Therefore, denoting $\M_\alpha^{(k)} = [m_{ij}^{(k)}]$ and $\M_\alpha^{(k+1)} = [m_{ij}^{(k+1)}]$, we take now as stopping criterion that
\begin{equation}
\label{e:epsmijk}
\max_{ij}\left(\min\left\{\left|m_{ij}^{(k)} - m_{ij}^{(k+1)}\right|, \left|\frac{m_{ij}^{(k)} - m_{ij}^{(k+1)}}{m_{ij}^{(k+1)}}\right| \right\}\right) < \varepsilon,
\end{equation}

\noindent where, again, $\varepsilon = 2^{-52}$. We have found this criterion to be very adequate, because, when $m_{ij}^{(k)}$ and $m_{ij}^{(k + 1)}$ are infinitesimal, the absolute discrepancy is enough to establish convergence, whereas, for larger values, it is preferable to consider the relative discrepancy.

Similarly as in the previous example, in order to obtain the number of digits for a given $N + 1$, we start generating the matrix $\M_{\sqrt3}^{(k)}\in\mathcal M_{(N+1)\times(N+1)}(\mathbb C)$ with the number of digits $k$ necessary to generate $\M_{\sqrt3}^{(k)}\in\mathcal M_{N\times N}(\mathbb C)$ satisfying \eqref{e:epsmijk}, and then increase $k$ by one unity, until \eqref{e:epsmijk} is fulfilled. On the left-hand side of Figure \ref{f:digitsMMa}, we have plotted the number of digits with respect to $N$; the graph is very similar (but not identical) to that in Figure \ref{f:digitsMM0}, the least squares regression line being now $\hat y = 0.47777x + 7.3617$; and we have found very similar results for other values of $\alpha$. Obviously, if we consider another nodal distribution such that every node $x_j$ satisfies $2 / |ix_j + 1| \ll 1$, the number of digits will be smaller.
\begin{figure}[!htbp]
	\centering
	\includegraphics[width=0.5\textwidth, clip=true]{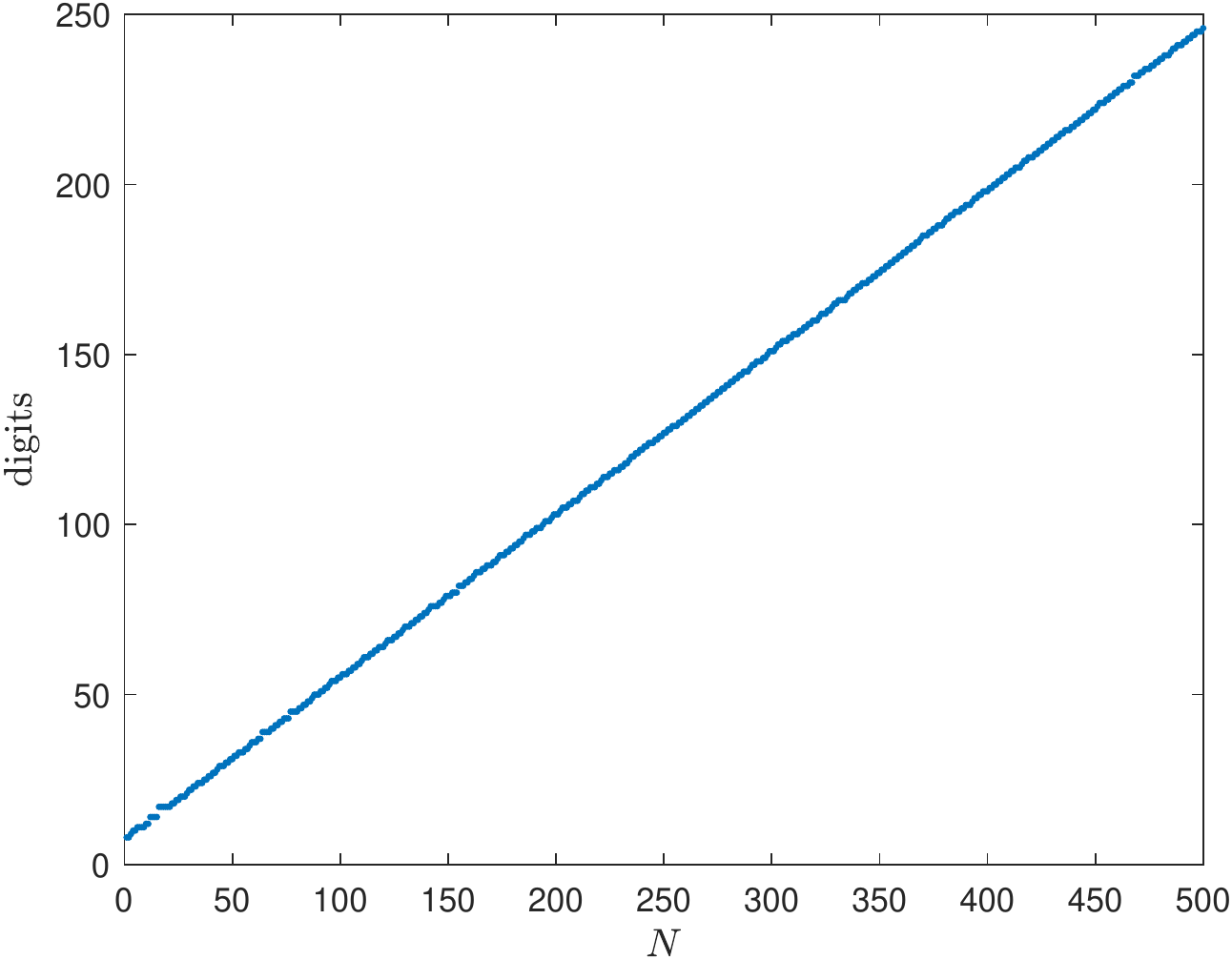}\includegraphics[width=0.5\textwidth, clip=true]{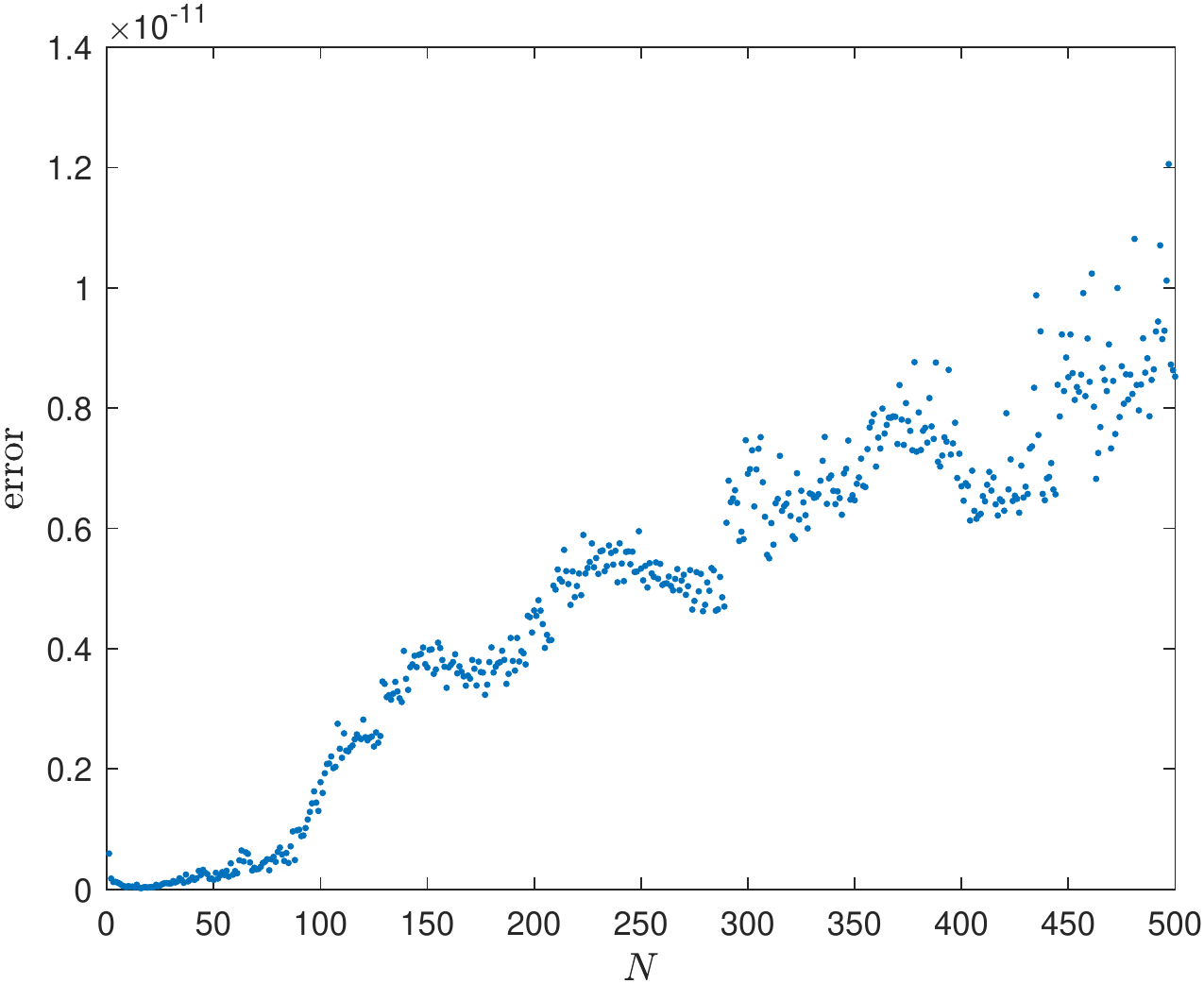}
	\caption{Left: Number of digits necessary to approximate $\M_\alpha^{vpa}\in\mathcal M_{M\times M}(\mathbb C)$, as a function of $M$, for $\alpha = \sqrt 3$. Right: Comparison with the results in \cite{cayamacuestadelahoz2019}, using \eqref{e:discrepancy}.}
	\label{f:digitsMMa}	
\end{figure}

We have also compared the approximation of the $N^2$ values $(-\Delta)^{(\sqrt3/2)}\lambda_n(x_j)$ via the matrices just generated, which we denote $\M_\alpha^{vpa}\in\mathcal M_{N\times N}(\mathbb C)$, with those given by the method explained in \cite{cayamacuestadelahoz2019}, which we denote $\M_\alpha^{old}\in\mathcal M_{N\times N}(\mathbb C)$; remark that we have adopted the method in \cite{cayamacuestadelahoz2019} with only minor modification to make it produce matrices of the required size, using $l_{lim}=1000$, for all the values of $N$. On the right-hand side of Figure \ref{f:digitsMMa}, we have plotted the discrepancy $\operatorname{d}(\M_\alpha^{vpa}, \M_\alpha^{old})$ between both (largely unrelated) techniques, using a formula similar to \eqref{e:epsmijk}:
\begin{equation}
\label{e:discrepancy}
\operatorname{d}(\M_\alpha^1, \M_\alpha^2) \equiv \max_{ij}\left(\min\left\{\left|m_{ij}^1 - m_{ij}^2\right|, \left|\frac{m_{ij}^1 - m_{ij}^2}{m_{ij}^1}\right| \right\}\right);
\end{equation}

\noindent such discrepancy is of the order of $10^{-11}$ for the first $500$ values of $N$. Note that, although \eqref{e:discrepancy} is not symmetric, the results provided by $\operatorname{d}(\M_\alpha^{vpa}, \M_\alpha^{old})$ and $\operatorname{d}(\M_\alpha^{old}, \M_\alpha^{vpa})$ show only infinitesimal variations in our case. We have also generated $\M_\alpha^{vpa}$ and $\M_\alpha^{old}$, for $\alpha \in\{0.01, 0.02, \ldots, 1.99\}$, working with 250 digits in the case of $\M_\alpha^{vpa}$; in Figure \ref{f:errorMaMa2alla}, we have plotted $\operatorname{d}(\M_\alpha^{old}, \M_\alpha^{vpa})$ as a function of $\alpha$, which is again of the order of $10^{-11}$. These results confirm the validity of both approaches when approximating the fractional Laplacian.
\begin{figure}[!htbp]
	\centering
	\includegraphics[width=0.5\textwidth, clip=true]{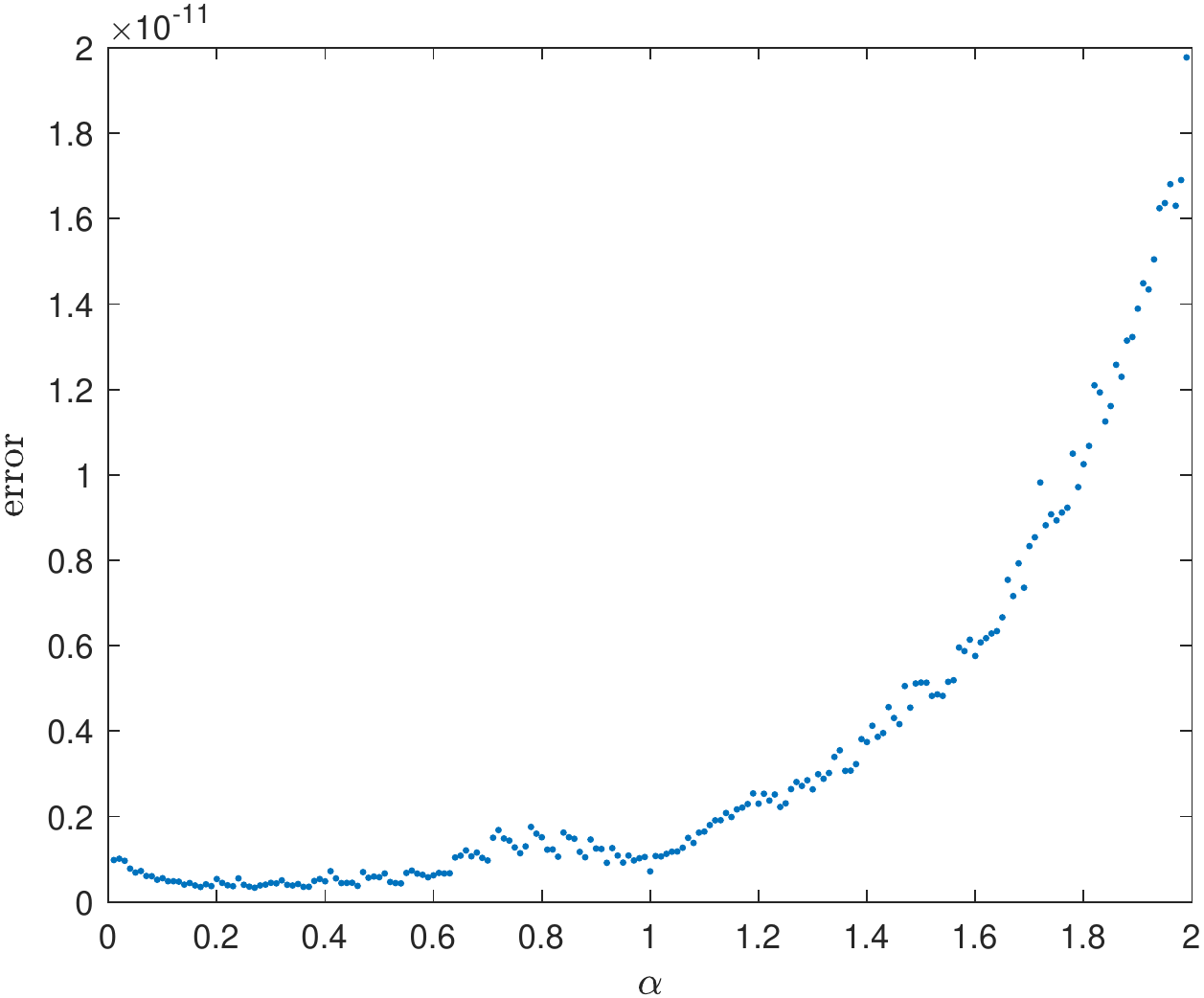}
	\caption{Comparison with the results in \cite{cayamacuestadelahoz2019}, using \eqref{e:discrepancy}, for $N = 500$ and different values of $\alpha$.}
	\label{f:errorMaMa2alla}
\end{figure}

The method that we are describing in this paper is, in comparison with the method described in \cite{cayamacuestadelahoz2019}, much simpler to implement and \textit{exact} by definition, but, in principle, much more expensive computationally, even if it still largely outperforms the Matlab function \texttt{hypergeom}. Recall that \texttt{hypergeom} uses transformation formulas like those in \cite[Ch. 15.3]{abramowitz}, which allow it to evaluate ${}_2F_1$ in \eqref{e:fraclapln}, using 64-bit floating point arithmetic, for a much larger range of values of $n$ than it would be possible by simply using the representation \eqref{hyper:2}, although it has to be invoked individually for each value of $n$. To give an idea of the time needed for each technique, we have considered $\M_\alpha\in\mathcal M_{N\times N}(\mathbf C)$, for $N = 330$, i.e., a size for which \texttt{hypergeom} produces accurate results without using arbitrary precision. We have generated column-wise $\M_\alpha$ with \texttt{hypergeom} using directly \eqref{e:fraclapln}, which we denote $\M_\alpha^{dp}$, where $dp$ stands for double (i.e., 64-bit) precision, needing 1265.04 sec, whereas $\M_\alpha^{old}$ (based on \cite{cayamacuestadelahoz2019}) and $\M_\alpha^{vpa}$ (using \texttt{vpa}) required only 3.44 sec and 54.64 sec, respectively.

To have a more complete picture of the possibilities and limitations of variable-precision arithmetic, we have also tested \textsc{Advanpix} \cite{advanpix}, the Multiprecision Computing Toolbox for \textsc{MATLAB}. Its main command is \texttt{mp}, which stands for multiple precision, and works pretty much in the same way as \texttt{vpa}, but, unlike \texttt{vpa}, it cannot receive symbolic objects as parameters; e.g., to obtain the multiply precision approximation of $\sqrt 3$, we just type \texttt{mp('sqrt(3)')}. By default, the number of significant digits of \texttt{mp} is 34 (quadruple precision); and it can changed globally to $d$ digits by means of \texttt{mp.Digits(d)}, or individually, as a second argument of \texttt{mp}. It is important to remark, however, that \texttt{vpa} and \texttt{mp} do not have exactly the same behavior; e.g, both \texttt{vpa(str2sym('sqrt(3)'), 2)} and \texttt{mp('sqrt(3)', 2)} show $1.7$ at the screen, but \texttt{double(vpa(str2sym('sqrt(3)'), 2)) - sqrt(3)} gives $1.1206\cdot10^{-10}$, whereas \texttt{double(mp('sqrt(3)', 2)) - sqrt(3)} gives $0.0023$. This phenomenon can explained because \texttt{vpa(str2sym('sqrt(3)'), 2)} guarantees at least two digits, even if more digits can be used internally, whereas \texttt{mp('sqrt(3)', 2)} really yields two digits.

Using \textsc{Advanpix}, we have generated again $\M_\alpha\in\mathcal M_{N\times N}(\mathbb C)$ for $\alpha = \sqrt 3$, which we denote $\M_\alpha^{np}$; the code is identical to that of $\M_\alpha^{vpa}$, except that we have replaced the three appearances of \texttt{vpa} with \texttt{mp}, namely \texttt{mp('pi')}, \texttt{mp('sqrt(3)')} and \texttt{mp(zeros(N, N))}. In Figure \ref{f:digitsMMmpa}, we have plotted the number of digits used to generate $\M_\alpha^{mp}$, and the least squares regression line is now $\hat y = 0.47686x + 14.708$. Therefore, for a given $N$, $\M_\alpha^{mp}$ requires some seven or eight more digits than $\M_\alpha^{vpa}$. On the other hand, the plot of $\operatorname{d}(\M_\alpha^{mp}, \M_\alpha^{old})$ is virtually identical to that on the right-hand side of Figure \ref{f:digitsMMa}, except for infinitesimal variations of the order of $10^{-16}$ for a few values of $N$.
\begin{figure}[!htbp]
	\centering
	\includegraphics[width=0.5\textwidth, clip=true]{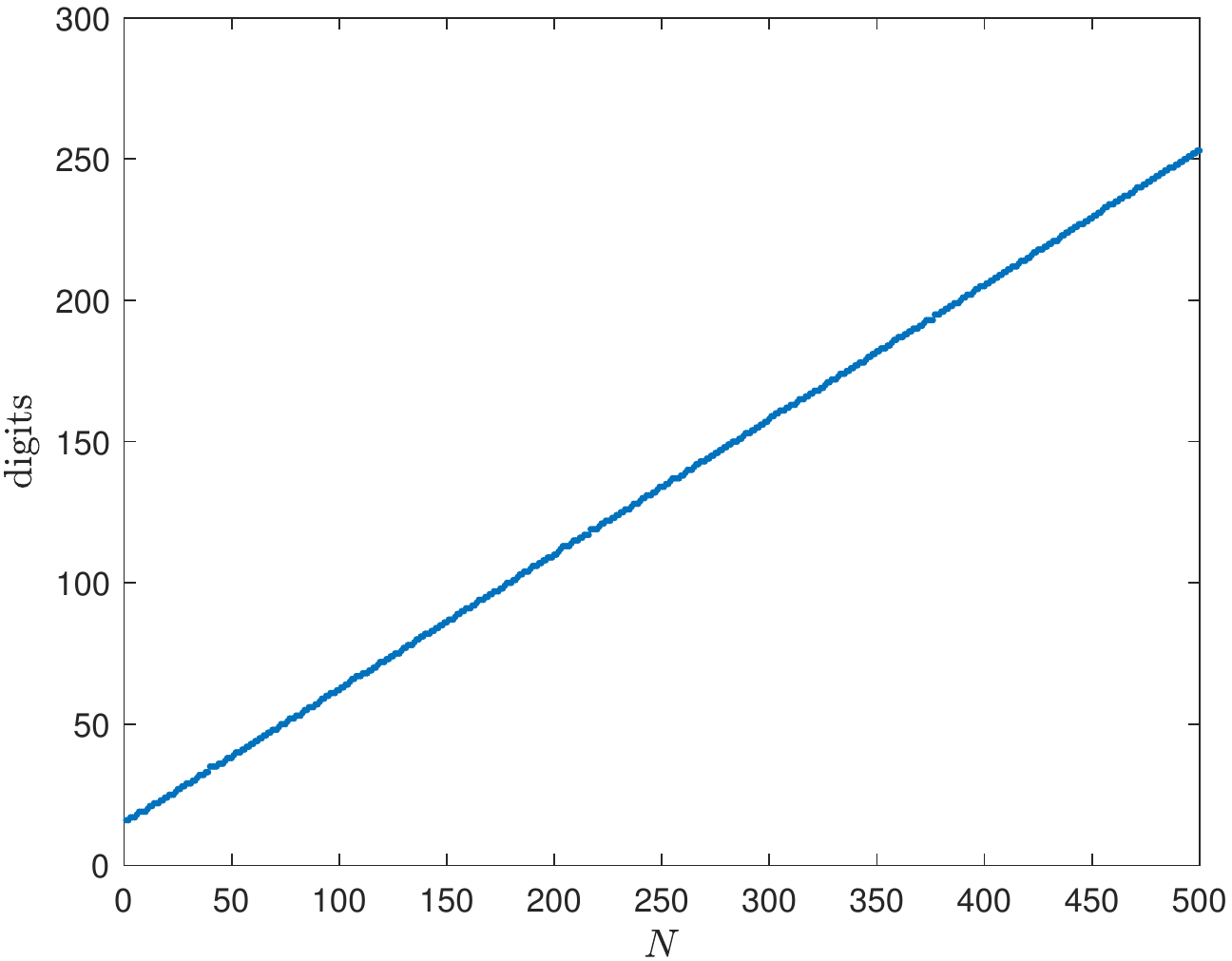}
	\caption{Number of digits necessary to approximate $\M_\alpha^{mp}\in\mathcal M_{M\times M}(\mathbb C)$, as a function of $M$}
	\label{f:digitsMMmpa}
\end{figure}

In Table \ref{t:times}, we show the times required to generate $\M_\alpha\in\mathcal M_{M\times M}(\mathbb C)$ using different techniques and sizes. These times are provided on a purely indicative basis, because they change slightly from execution to execution, and, probably, all of them can be improved by compiling the codes, etc. In our opinion, the results are quite revealing, allowing us to make some conclusions; indeed, for $N$ moderately large (e.g., $N = 500$), the method described in this paper, together with \textsc{Advanpix}, is absolutely competitive; and reasonably fast, when combined with \texttt{vpa}. Furthermore, \textsc{Advanpix}, can be still recommended to generate larger matrices (e.g., $N = 1000$). On the other hand, even if, in principle, it is possible to generate very large matrices (e.g., $N = 1500$, $N = 2000$ or even larger sizes), it may be advisable to use in those case the method described in \cite{cayamacuestadelahoz2019}, except when time is not an issue and the matrix needs to be generated only once or a few times. Finally, let us mention that the values of the discrepancy \eqref{e:discrepancy} between $\M_\alpha^{old}$ and $\M_\alpha^{np}$ for $N = 1000$, $N = 1500$ and $N = 2000$ are respectively $2.1238\cdot10^{-11}$, $5.1753\cdot10^{-11}$, and $8.2190\cdot10^{-11}$, i.e., of the order of $10^{-11}$.
\begin{table}[!htbp]
	\centering
	\begin{tabular}{|c||c||c|c||c|c||c|}
		\hline $N$ & $\M_\alpha^{old}$ & $\M_\alpha^{mp}$ & digits & $\M_\alpha^{vpa}$ & digits & $\M_\alpha^{dp}$ 
		\\
		\hline $330$ & $3.38$ & $6.18$ & $172$ & $54.64$ & $165$ & $1265.04$
		\\
		\hline $500$ & $7.84$ & $21.41$ & $253$ & $205.12$ & $246$
		\\
		\cline{1-6} $1000$ & $33.05$ & $237.10$ & $500$ & $2272.98$ & $500$
		\\
		\cline{1-6} $1500$ & $80.56$ & $1010.76$ & $750$ & $9922.05$ & $750$
		\\
		\cline{1-6} $2000$ & $158.01$ & $3422.38$ & $1000$
		\\
		\cline{1-4}
	\end{tabular}
	\caption{Elapsed time in seconds used for the generation of $\M_\alpha\in\mathcal M_{N\times N}(\mathbb C)$, for different values of $N$ and different techniques. $\M_\alpha^{old}$ corresponds to the method described in \cite{cayamacuestadelahoz2019}; $\M_\alpha^{mp}$ uses the \textsc{Advanpix} function \texttt{mp}; $\M_\alpha^{vpa}$ uses \texttt{vpa}; and $\M_\alpha^{vpa}$ uses \texttt{hypergeom} without arbitrary precision. In the case of $\M_\alpha^{mp}$ and $\M_\alpha^{vpa}$, the number of digits to generate the matrices is also offered.}\label{t:times}
\end{table}

\section*{Acknowledgments}

The authors acknowledge the financial support of the Spanish Government through the MICINNU project PGC2018-094522-B-I00, and of the Basque Government through the Research Group grant IT1247-19. J. Cayama also acknowledges the support of the Spanish Government through the grant BES-2015-071231.


\begin{thebibliography}{10}
	
	\bibitem{Bueno-OrovioKayBurrage2012}
	D.~Kay A.~Bueno-Orovio and K.~Burrage.
	\newblock Fourier spectral methods for fractional-in-space reaction-diffusion
	equations.
	\newblock {\em Journal of Computational Physics}, 2012.
	
	\bibitem{abramowitz}
	M.~Abramowitz and I.~A. Stegun.
	\newblock {\em {Handbook of Mathematical Functions}}.
	\newblock Dover, 1972.
	\newblock 10th printing with corrections.
	
	\bibitem{advanpix}
	{Advanpix LLC.}
	\newblock {Multiprecision Computing Toolbox for MATLAB, Version 4.7.0.13560}.
	\newblock http://www.advanpix.com, 2019.
	
	\bibitem{boyd1990}
	J.~P. Boyd.
	\newblock {The Orthogonal Rational Functions of Higgins and Christov and
		Algebraically Mapped Chebyshev Polynomials}.
	\newblock {\em Journal of Approximation Theory}, 61(1):98--10, 1990.
	
	\bibitem{boydxu2011}
	J.~P. Boyd and Z.~Xu.
	\newblock {Comparison of three spectral methods for the Benjamin-Ono equation:
		Fourier pseudospectral, rational Christov functions and Gaussian radial basis
		functions}.
	\newblock {\em Wave Motion}, 48(8):702--706, 2011.
	
	\bibitem{cayamacuestadelahoz2019}
	J.~Cayama, C.~M. Cuesta, and F.~de~la Hoz.
	\newblock {A Pseudospectral Method for the One-Dimensional Fractional Laplacian
		on $\mathbb R$}.
	\newblock {\em arXiv:1908.09143 [math.NA]}, 2019.
	
	\bibitem{christov}
	C.~I. Christov.
	\newblock {A Complete Orthonormal System of Functions in $L^2(-\infty,\infty )$
		Space}.
	\newblock {\em SIAM Journal on Applied Mathematics}, 42(6):1337--1344, 1982.
	
	\bibitem{higgins}
	J.~R. Higgins.
	\newblock {\em {Completeness and Basis Properties of Sets of Special
			Functions}}.
	\newblock Cambridge University Press, 1977.
	
	\bibitem{hilbert}
	D.~Hilbert.
	\newblock {\em {Grundz\"uge einer allgemeinen Theorie der linearen
			Integralgleichungen}}.
	\newblock Fortschritte der Mathematischen Wissenschaften in Monographien. Druck
	und Verlag von B.~G. Teubner, Leipzig und Berlin, 1912.
	\newblock In German.
	
	\bibitem{IlicLiuTurnerAnh2005}
	M.~Ilic, F.~Liu, I.~Turner, and V.~Anh.
	\newblock Numerical approximation of a fractional-in-space diffusion equation.
	\newblock {\em Fractional Calculus and Applied Analysis, An International
		Journal for Theory and Applicatios}, 6(3), 2005.
	
	\bibitem{johansson}
	F.~Johansson.
	\newblock {Computing Hypergeometric Functions Rigorously}.
	\newblock {\em ACM Transactions on Mathematical Software (TOMS)}, 45(3):30,
	2019.
	
	\bibitem{kwasnicki}
	M.~Kwa\'snicki.
	\newblock {Ten equivalent definitions of the Fractional Laplace Operator}.
	\newblock {\em Fract. Calc. Appl. Anal.}, 20(1):7--51, 2017.
	
	\bibitem{lischke}
	A.~Lischke~et al.
	\newblock {What is the Fractional Laplacian?}
	\newblock {\em arXiv:1801.09767v1 [math.NA]}, 2018.
	
	\bibitem{narayan}
	A.~C. Narayan and J.~S. Hesthaven.
	\newblock {A generalization of the Wiener rational basis functions on infinite
		intervals: Part I–derivation and properties}.
	\newblock {\em Mathematics and Computation}, 80(275):1557--1583, 2010.
	
	\bibitem{YangLiuTurner2010}
	F.~Liu Q.~Yang and I.~Turner.
	\newblock Numerical methods for fractional partial differential equations with
	riesz space fractional derivatives.
	\newblock {\em Applied Mathematical Modelling}, 34:200--218, 2010.
	
	\bibitem{matlab}
	{The MathWorks Inc.}
	\newblock {MATLAB, Version R2019b}.
	\newblock https://www.mathworks.com, 2019.
	
	\bibitem{weideman1995}
	J.~A.~C. Weideman.
	\newblock {Computing the Hilbert transform on the real line}.
	\newblock {\em Mathematics of Computation}, 64(210):745--762, 1995.
	
	\bibitem{wiener}
	N.~Wiener.
	\newblock {\em {Extrapolation, Interpolation, and Smoothing of Stationary Time
			Series}}.
	\newblock M.I.T. Press Paperback Series (Book 9), 1964.
	
	\bibitem{mathematica}
	{Wolfram Research, Inc.}
	\newblock {Mathematica, Version 11.3}.
	\newblock https://www.wolfram.com, 2018.
	
\end{thebibliography}
\end{document}